\documentclass[intlim,righttag,10pt,openbib]{amsart}
\usepackage{enumitem}
\usepackage{amscd}
\usepackage{amssymb}
\usepackage{graphicx} 
\usepackage[all]{xy}
\usepackage{fge}

\usepackage{color}
\usepackage{tikz-cd}

\usepackage{todonotes}

\RequirePackage[T1]{fontenc}
\RequirePackage{amsfonts,latexsym,amssymb}

\RequirePackage{mathrsfs}
\let\mathcal\mathscr
\let\cal\mathcal

\let\bb\mathbb

\usepackage{bbm}
\usepackage{hyperref}

\newtheorem{theorem}[equation]{Theorem}
 \newtheorem{lemma}[equation]{Lemma}
 \newtheorem{proposition}[equation]{Proposition}
 \newtheorem{corollary}[equation]{Corollary}

\theoremstyle{definition}
\newtheorem{definition}[equation]{Definition}
\newtheorem{notation}[equation]{Notation}

\newtheorem{remark}[equation]{Remark}
\newtheorem{example}[equation]{Example}
\theoremstyle{remark}
\newtheorem*{acknowledgments}{Acknowledgments}

\def\invlim{\mathop{\vtop{\ialign{##\crcr$\hfill{\lim}\hfil$\crcr
\noalign{\kern1pt\nointerlineskip}\leftarrowfill\crcr\noalign
{\kern -3pt}}}}\limits}
\def\dirlim{\mathop{\vtop{\ialign{##\crcr$\hfill{\lim}\hfil$\crcr
\noalign{\kern1pt\nointerlineskip}\rightarrowfill\crcr\noalign
{\kern -3pt}}}}\limits}

\def\phi{\varphi}
\def\epsilon{\varepsilon}

\newcommand{\Spa}{\operatorname{Spa}}

\newcommand{\ovk}{\overline{K} }

\newcommand{\dr}{\operatorname{dR} }

\newcommand{\proeet}{\operatorname{pro\acute{e}t} }

\newcommand{\eet}{\operatorname{\acute{e}t} }
\newcommand{\Spec}{\operatorname{Spec} }

\newcommand{\Hom}{{\rm{Hom}} }

\newcommand{\st}{\operatorname{st} }

\newcommand{\coker}{\operatorname{coker} }

\newcommand{\im}{\operatorname{Im} }

\newcommand{\so}{{\mathcal O}}

\newcommand{\wh}{\widehat}

\DeclareMathOperator{\hhh}{H}
\DeclareMathOperator{\rrr}{R}

\numberwithin{equation}{section}

\def\R{{\mathrm R}}

 \def\A{{\bf A}} \def\B{{\bf B}}
\def\Q{{\bf Q}} \def\Z{{\bf Z}}

\def\N{{\bf N}}

\def\G{{\cal G}}

\def\limi{\varinjlim}
\def\limp{\varprojlim}

\def\epsilon{\varepsilon}

\def\BC{{\mathcal{B}}}
\def\CC{{\mathcal{C}}}

\def\FC{{\mathcal{F}}}

\def\OC{{\mathcal{O}}}

\def\TC{{\mathcal{T}}}

\numberwithin{equation}{section}
\begin{document}
\title[${\bb G}_m$-cohomology of $p$-adic Stein spaces]{${\bb G}_m$-cohomology of $p$-adic Stein spaces}
\author{Sally Gilles}
\address{Universit\"at Duisburg-Essen, Fakult\"at f\"ur Mathematik, Thea-Leymann-Str. 9, 
45127 Essen, Germany}
\email{sally.gilles@uni-due.de}
\author{Damien Junger} 
\address{Laboratoire de Mathématiques de Besançon, CNRS UMR 6623, 16 route de Gray, 25030 Besançon cedex, France}
\email{damien.junger@umlp.fr} 
\date{\today}

\maketitle

 \begin{abstract}
We compute the \'etale ${\bb G}_m$-cohomology of some $p$-adic rigid analytic Stein spaces. The computation is done by considering the filtration induced by the subgroup of principal units $U=1+ \mathfrak{m} \mathcal{O}^+$ of ${\bb G}_m$. We then determine the $U$-cohomology via methods from $p$-adic Hodge theory (passage to the pro-\'etale site, comparison theorems with $p$-adic cohomologies), while the ${\bb G}_m/U$-cohomology is obtained using Kummer exact sequences. In particular, our formula applies to the case of Drinfeld upper-half space.  
 \end{abstract}
\tableofcontents

\section{Introduction}

Our aim in this work is to describe the cohomology of ${\bb G}_m$, the \'etale sheaf of invertible functions, for some classical rigid spaces and more precisely smooth Stein spaces over a complete algebraically closed field $C/\Q_p$ (to avoid the contribution of the base fields and other technical descent). One of the main interests of these groups comes from the fact that they admit geometric interpretations, more particularly in degrees $0$, $1$ and $2$. For example, if a space $X$ is the complement of a normal crossing divisor $F$ on a proper space $\bar{X}$ (i.e. $X=\bar{X}\setminus F$), the group $\hhh^0_{\eet}(X, {\bb G}_m)$ should reflect some aspects of  the geometry of its compactification $\bar{X}$. On the other hand, the first degree is the Picard group which classifies some locally free modules of rank $1$ over the structure sheaf $\so$, whereas the second degree  classifies, at least in the algebraic case, some Azumaya algebras and Severi-Brauer varieties and these groups constitute deep invariants of the studied space. 

Those cohomology groups should also carry some motivic aspects of the space in question. For example, we have as many $\ell$-adic cohomology groups as prime integers $\ell$, and they all live over different fields with different structures and topologies. The sheaf ${\bb G}_m$  on the other hand only admits a structure of $\Z$-module but it can be related to all of these groups at the same time  by the Kummer exact sequence suggesting the hope to exhibit a common  $\Z$-structure.  

A lot is known for the  ${\bb G}_m$-cohomology for algebraic spaces, but the situation is rather mysterious on the analytic case, except in degree $0$ or $1$, where they could be compared to the cohomology on the analytic site (see for example for more detail on this strategy \cite{vdp,J1}). However, many phenomena that can only be seen on the \'etale site vanish in higher degrees and new methods are needed  (for example, in most of the examples treated in previous cited works, the analytic cohomology of ${\bb G}_m$ will vanish in degree higher than $2$ whereas \'etale cohomology will still exhibit interesting representations). 

Let us say here a few words about the algebraic case from which we will draw our inspiration for the strategy of this paper. Here the key result is the following observation of Grothendieck.

\begin{theorem}\cite[Proposition 1.4]{grobrau2}
For a regular noetherian scheme $X/C$, the groups $\hhh^i_{\eet}(X, {\bb G}_m)$ are torsion in degree $i\ge 2$. 
\end{theorem}

This allows, thanks to the Kummer exact sequence, to completely control the cohomology of ${\bb G}_m$ by the $\ell$-adic cohomology for all prime $\ell$. More precisely if the groups $\hhh^i_{\eet}(X, \Z/\ell^n\Z(-1))$ are all free over $\Z/\ell^n\Z$ of rank $\alpha_i$ independent of the prime $\ell$ and the integer $n$, then we would expect an identity of the form \[\hhh^i_{\eet}(X, {\bb G}_m)\simeq\bigoplus_{\ell}(\Q_\ell/\Z_\ell)^{\alpha_i}\simeq(\Q/\Z)^{\alpha_i}.\] This is best illustrated in the classical formula for  a point $X=\Spec(K)$ (more precisely when $K/\Q_p$ is a finite extension) \[{\rm Br}(K)\simeq  \Q/\Z\] proved via Hilbert's Theorem 90 and the calculation of the $\ell$-adic cohomology of $K$ (even if this particular example has preceded historically and suggested this strategy).

The idea is to exhibit a presentation of ${\bb G}_m$ \[0\to {\bb G}_m \to \iota_{\eta,*}K(\eta)\to {\rm Div}\to 0\] where each $\iota_{\eta,*}K(\eta)$ and ${\rm Div}$ can be related to pushforwards of sheaves on points of $X$. But the cohomology of these sheaves are known to be  torsion on non-zero degrees by  classical results on  cohomology of profinite groups. Unfortunately, this strategy completely fails for analytic spaces where we could not have such a simple description of ${\bb G}_m$ in terms of points. 

Even worse, we can exhibit smooth spaces for which these groups are not torsion. This is mainly due to the fact that the category of torsion groups is not stable under infinite limit in the category of groups. In particular, when we are interested in a space $X=\bigcup_n X_n$ admitting a presentation by increasing affinoid open $X_n$ (typically a Stein space), then the group $\hhh^i_{\eet}(X, {\bb G}_m)$ should admit as a quotient $\varprojlim_{n}\hhh^i_{\eet}(X_n, {\bb G}_m)$, which could very well be not torsion even when all the groups $\hhh^i_{\eet}(X_n, {\bb G}_m)$ are themselves torsion. It will be indeed the case for our main example, the Drinfeld symmetric space. However in  the examples we consider, all the desired cohomology groups are torsion for every opens $X_n$ making the obstruction still controlled and manageable. 

Our hope is to exhibit a sufficiently large class of  smooth quasicompact spaces for which those cohomology groups are indeed torsion. The first intermediate result of this paper actually builds on this intuition. To  properly state it, we introduce the usual Stein covering of the Drinfeld space ${\bb H}^d=\bigcup_n U_n$ of dimension $d \ge 1$ over $C$ (see for example \cite[§1]{scst}).

\begin{proposition}(Corollary \ref{torsion-disc1})
The  closed balls $\B(r)$, the closed annuli $\A(r,s)$ and the opens $U_n$ have  torsion  ${\bb G}_m$-cohomology in degree $i\ge 2$.
\end{proposition}

We say a few words about the arguments. The sheaf   ${\bb G}_m$ admits, as a subgroup, the group  of principal units $U:=1+\so^{++}\subset {\bb G}_m$ which consists of the elements of the form $1+f$ with $f$ topologically nilpotent. Denoting by $\overline{\so}^{\times}={\bb G}_m/U$ the quotient,  both groups $U$ and $\overline{\so}^{\times}$ have intimate relations with coherent groups thanks to the following $\log$ and $\exp$ exact sequences  \begin{align*}
& 1\to \mu_{p^{\infty}} \to U \xrightarrow{{\rm log}}  \so \to 1, \\ 
& 1\to \so \xrightarrow{{\rm exp}} \so^{\times}[\tfrac{1}{p}] \to {\overline{\so}}^{\times} \to 1. \notag
  \end{align*}  
The long exact sequence associated to $1\to U\to{\bb G}_{m}\to \overline{\so}^{\times}\to 1$ allows to focus on studying the torsion of both groups $U$ and $\overline{\so}^{\times}$. The first one is automatically torsion when the space is quasi-compact and has no coherent cohomology thanks to the $\log$ exact sequence, and we are reduced to study  $\overline{\so}^{\times}$. For the moment, we can only show the desired  torsion property for specific examples, namely the balls, the annuli and other spaces built around these two such as the opens $U_n$. On the other hand,   $\overline{\so}^{\times}$ is related to the invertible functions on the special fiber when there exists a nice enough model and our hope is that in the general case (at least the semi-stable one), this question  could be treated by reduction to the algebraic case. This is somewhat close to what we did for the balls and the annuli. Actually, we have shown this intermediate property.

\begin{proposition}\label{torsijntro}(Lemma \ref{homot})
If $\hhh^i_{\eet}(X, \overline{\so}^{\times})$ is torsion then so are $\hhh^i_{\eet}(X \times \B(r), \overline{\so}^{\times})$ and $\hhh^i_{\eet}(X \times \A(r,s), \overline{\so}^{\times})$ for any smooth $X$.
\end{proposition}


The argument presented here actually reveals some remarkable phenomenon as the group $U$ only exhibits the $p^\infty$-torsion of $\hhh^i_{\eet}(X, {\bb G}_m)$, whereas $\overline{\so}^{\times}$ is related to the prime-to-$p$ torsion (see Lemma \ref{U-filtration1} and Remark \ref{ptors}). Thanks to  this intermediate result, we can determine the cohomology of each open affinoid $X_n$ in the Stein covering of $X$ by the Kummer exact sequence, and then apply a limit argument to  completely describe it as an extension of $\limp_{n}\hhh^i_{\eet}(X_n, {\bb G}_m)$ by  $\rrr^1\limp_{n}\hhh^{i-1}_{\eet}(X_n, {\bb G}_m)$ by the usual formalism of the derived limit. Thanks to this strategy, we are able to describe $\hhh^i_{\eet}(X, {\bb G}_m)$ for some classical Stein spaces and obtain the main result of our paper. 

\begin{theorem}\label{mainthintro}
We have, for the affine space ${\bb A}^d$ of dimension $d \ge 1$ over $C$, \[ \hhh^i_{\eet}({\bb A}^d, {\bb G}_m) \simeq \Omega^{i-1}({\bb A}^d)(-i+1)/{\rm Ker}(d), \]  and for $i\ge 2$ and $X$ a poly-torus ${\bb G}_m^d$  or the Drinfeld symmetric space ${\bb H}$, there exists a filtration $\FC^{(i)}(X)$ of $\hhh^i_{\eet}(X, {\bb G}_m)$  such that the graded pieces are:
\begin{align*}
{\rm gr}^0 \FC_{\eet,{\bb G}_{m}}^{(i)}({\bb G}_{m,C}^d) &\simeq \Omega^{i-1}({\bb G}^d_{m,C})(-i+1)/{\rm Ker}(d), \ \ \ {\rm gr}^1 \FC_{\eet,{\bb G}_m}^{(i)}({\bb G}_{m,C}^d) \simeq (\wedge^{i} \Q^d_p/\wedge^{i} \Z^d_p)(-i+1), \\ 
&{\rm gr}^2 \FC_{\eet,{\bb G}_m}^{(i)}({\bb G}_{m,C}^d) \simeq \bigoplus_{\ell\neq p} (\wedge^{i} \Q^d_\ell/\wedge^{i} \Z^d_\ell)(-i+1); 
\end{align*}
\begin{align*}
{\rm gr}^0 \FC_{\eet,{\bb G}_m}^{(i)}({\bb H}^d) &\simeq \Omega^{i-1}({\bb H}^d)(-i+1)/{\rm Ker}(d), \ \ \ {\rm gr}^1 \FC_{\eet,{\bb G}_m}^{(i)}({\bb H}^d) \simeq  ({\rm Sp}_{i}(\Q_p)^*/{\rm Sp}_{i}( \Z_p)^*)(-i+1), \\
&{\rm gr}^2 \FC_{\eet,{\bb G}_m}^{(i)}({\bb H}^d) \simeq \widehat{\bigoplus_{\ell\neq p}} (({\rm Sp}_{i}(\Z_\ell)^* \otimes \Q_{\ell})/{\rm Sp}_{i}(\Z_\ell)^*)(-i+1),
\end{align*}
 where the sum is completed for some well-chosen topology (see \textsection~\ref{l-adic} for the definition of this topology).
\end{theorem}

\begin{remark}
\begin{itemize}
    \item We can actually  construct a filtration $\FC^{(i)}(X)$ of $\hhh^i_{\eet}(X, {\bb G}_m)$ for any  smooth Stein space $X/C$.
    \item For the affine space, the first and second graded pieces $\FC^{(i)}({\bb A}^d)$ are trivial and all three descriptions above are actually analogous.
    \item The graded pieces of this filtration are not explicit for  any  smooth Stein space, but we give general conditions for which we can  achieve the calculation. All the examples ${\bb A}^d$, ${\bb G}_{m}^d$ and ${\bb H}$ above satisfy these hypotheses. 
    \item Fix $(X,(X_n)_n))$ a smooth Stein space together with a fixed Stein covering and  let us put specify the three main conditions metnioned in the previous point (the rest of hypotheses are there to avoid some technicalities and pathologies in the limit argument), namely we ask for the groups $\hhh^i_{\eet}(X_n, {\bb G}_m)$ to be torsion, for the groups $\hhh^i_{\eet}(X_n, \Z/\ell^k\Z)$ with $\ell \neq p$ to be flat over $\Z/\ell^k\Z$ and for the injectivity of the maps $\hhh^i_{\eet}(X, \Z_p(i)) \to {\rm HK}^i(X)$ for all $n$ and $i\ge 2$, where $ {\rm HK}^i(X)$ denotes the $\wh{\B}^+_{\st}$-twisted Hyodo-Kato cohomology from Colmez-Dospinescu-Nizio{\l} (see Section~\ref{section-main-thm} for a more precise definition). 
    \item We can also construct a similar filtration for the pro-\'etale cohomology and calculate its graded pieces under the same general conditions. 
\end{itemize}
\end{remark}

Let us explain how to construct such a filtration $\FC^{(i)}(X)$ for a general smooth Stein space $X$. By d\'evissage of ${\bb G}_m$ by the principal units $U$, we have a filtration with subobject \[\FC_{1}^{(i)}(X):= {\rm Im}(\hhh^i_{\eet}(X, U)\to \hhh^i_{\eet}(X, {\bb G}_m))\] and quotient \[{\rm gr}^2 \FC^{(i)}(X):={\rm Im}(\hhh^i_{\eet}(X, {\bb G}_m)\to \hhh^i_{\eet}(X,\overline{\so}^{\times})).\] Note that by  the conclusion of Proposition~\ref{torsijntro} (together with extra technical conditions that we will not mention in this introduction, see for example Corollary \ref{U-filtration0} for more details), we have instead identifications $\FC_{ 1}^{(i)}(X)\simeq \hhh^i_{\eet}(X, U)$ and ${\rm gr}^2 \FC^{(i)}(X)\simeq \hhh^i_{\eet}(X, \overline{\so}^{\times})$.  Now we explain how to give a subfiltration of the term $\FC_{1}^{(i)}(X)$ by considering the pro-\'etale site. More precisely we have a natural map of topoi $\nu : X_{\proeet} \to X_{\eet}$ which gives rise to compatible  Leray spectral sequences \begin{align*}
&E_2^{i,j}({\bb G}_m)= \hhh^i_{\eet}(X, \R^j\nu_{*} {\bb G}_m) \Rightarrow \hhh^{i+j}_{\proeet}(X, {\bb G}_m)  \\
&E_2^{i,j}(U)= \hhh^i_{\eet}(X, \R^j\nu_{*} U) \Rightarrow \hhh^{i+j}_{\proeet}(X, U). \notag
\end{align*} Thanks to $\log$ and $\exp$ methods and Hodge-Tate isomorphism  most of the terms of those spectral sequences  vanish  by vanishing of the coherent cohomology of $X$. Moreover, for $i\ge 1$, we can construct (compatible) morphisms: 
\begin{align*}
    {\rm HTlog}_i : \hhh^i_{\proeet}(X, {\bb G}_m)\to \Omega^i(X)(-i) \\
    {\rm HTlog}_{U,i} : \hhh^i_{\proeet}(X, U)\to \Omega^i(X)(-i)
\end{align*}
such that the degeneracy of those spectral sequences yields (compatible) long exact sequences\begin{align*} \dots \xrightarrow{{\rm HTlog}_{i-1}} \Omega^{i-1}(X)(-i+1)\to \hhh^i_{\eet}(X, \so^{\times}) \to \hhh^i_{\proeet}(X, \so^{\times}) \xrightarrow{{\rm HTlog}_i} \Omega^i(X)(-i)\to \cdots, \\
\cdots \xrightarrow{{\rm HTlog}_{U,i-1}} \Omega^{i-1}(X)(-i+1) \to \hhh^i_{\eet}(X, U) \to \hhh^i_{\proeet}(X, U) \xrightarrow{{\rm HTlog}_{U,i}} \Omega^i(X)(-i)\to \cdots
\end{align*} Setting $\FC_{0}^{(i)}(X):= {\rm Im}(\Omega^{i-1}(X)\to \hhh^i_{\eet}(X, {\bb G}_m))\simeq{\rm Im}(\Omega^{i-1}(X)\to \hhh^i_{\eet}(X, U)) $ we can define a total filtration \[\hhh^i_{\eet}(X, {\bb G}_m)=\FC_{2}^{(i)}(X)\supset\FC_{1}^{(i)}(X)\supset\FC_{0}^{(i)}(X)\supset\FC_{-1}^{(i)}(X)=0.\]

Let us describe the different arguments used for the computation of these terms ${\rm gr}^k \FC^{(i)}(X)$. For $k=0, 1$, the graded pieces of interest identify with  $\ker({\rm HTlog}_{U,i+1})$ (respectively $\coker({\rm HTlog}_{U,i})$) which have been determined in \cite{EGN24}, by comparing the morphism ${\rm HTlog}_{U,i}$ with the Bloch-Kato exponential from the works of Colmez-Dospinescu-Nizio{\l} and Bosco. 

For $k=2$, we use Proposition \ref{torsijntro} and the strategy presented above to calculate $\hhh^i_{\eet}(X, \overline{\so}^{\times})$ and describe it in terms of the  $\ell$-adic cohomology groups for all prime $\ell\neq p$. Such a simple description is achieved when the space in question has divisible $\ell$-adic cohomology (and other technicalities). This can be  reduced to show that each group  $\hhh^i_{\eet}(X_n, \mu_{\ell^m})$ is flat over $\Z/\ell^m\Z$ for any $m>0$, $n$ and any  prime $\ell \neq p$, and with rank bounded "uniformly" in the prime $\ell$ (see hypothese ${\bf (H_2)}$ before the statement of the main theorem  \ref{main-theorem} for more details). For the balls $\B(r)$ and annuli $\A(r,s)$, this is relatively straightforward by explicit computations (and the fact that the transition maps in the projective system are isomorphisms). 

For the symmetric space ${\bb H}$, the task is far from being obvious because while an extensive literature exists on the calculation of the  $\ell$-adic cohomology groups for the global space ${\bb H}$, none of those articles describes the cohomology of the intermediate open $U_n$ appearing in the standard Stein covering of ${\bb H}$. We have managed to prove the desired flatness property by adapting some arguments of de Shalit \cite{ds} and some others by Schneider and Stuhler \cite{scst,scst2} for the global computation. These arguments  actually work for any open $U\subset  {\bb H}$ which is the preimage of a convex subcomplex in the associated Bruhat-Tits building. For $U_n$, it is the preimage of the ball of radius $n$ inside the Bruhat-Tits building for the  combinatorial distance, and this subcomplex happens to be convex\footnote{A statement not as straightforward as it seems!}.

Now we describe the content of the paper section by section. In section 2, we explain how to construct the filtration $\FC^{(i)}(X)$ of the main theorem for any smooth Stein space $X/C$ and state the main result, Theorem \ref{main-theorem}, where we give general technical hypotheses under which we can completely determine its graded pieces. Then we also state the main application, namely Theorem \ref{mainthintro}, which is the determination of the ${\bb G}_{m}$-cohomology of classical spaces. In the next section, we use methods from \cite{EGN24} to compute the zeroth and first graded pieces of $\FC^{(i)}(X)$ by relating the map ${\rm HTlog}$ and the Bloch-Kato exponential ${\rm Exp}$ under the hypothesis that the map $\iota :\hhh^i_{\eet}(X, \Z_p(i)) \to {\rm HK}^i(X)$ is injective. In section 4, we calculate the second graded piece when the ${\bb G}_{m}$-cohomology of all the opens $X_n$ of a fixed Stein covering of $X$ are torsion and all the groups $\hhh^i_{\eet}(X_n, \Z/\ell^k\Z)$ are flat over $\Z/\ell^k\Z$, thus finishing the proof of Theorem \ref{main-theorem}. The rest of the text is then devoted to prove that our three examples satisfy the general conditions mentioned above. For the injectivity of $\iota$ this can be checked directly by deep explicit computations available in the literature and section 5 establishes that all the classical examples have torsion $\overline{\so}^{\times}$-cohomology. This is done by standard calculations for ${\bb A}^d$ and ${\bb G}_{m}^d$ and then we use a combinatorial argument to propagate this property to  the Drinfeld symmetric space ${\bb H}$. The last section deals with the flatness of $\ell$-adic cohomology of some opens of ${\bb H}$ and it relies on the  acyclicity of coefficient systems on the Bruhat-Tits building coming from the theory of hyperplane arrangements.   

\begin{acknowledgments}
This paper generalizes a computation that was done in the case of the affine space in \cite{EGN24} by the first author together with Veronika Ertl and Wies{\l}awa Nizio{\l} and which was originally suggested by Ben Heuer: we would like to thank the three of them for helpful conversations related to this subject, as well as Frédéric Déglise, Gabriel Dospinescu, Philippe Gille and Jan Kohlhaase.

The second author's work has been funded by the Deutsche Forschungsgemeinschaft (DFG, German Research Foundation) under Germany's Excellence Strategy EXC 2044–390685587, Mathematics Münster: Dynamics–Geometry–Structure, and  support by the ERC in form of Consolidator Grant 770936: Newton-Strat. He is also a member of the project ANR-25-CE40-7869 "pDefi".
\end{acknowledgments}

\subsubsection*{Notation and conventions.} Let $\so_K$ be a complete discrete valuation ring with fraction field $K$  of characteristic 0 and with perfect
residue field ${\bf F}$ of characteristic $p$. Fix $\varpi$ a uniformizer. Let $\ovk$ be an algebraic closure of $K$ and let $\so_{\ovk}$ denote the integral closure of $\so_K$ in $\ovk$. Let $C=\wh{\ovk}$ be the $p$-adic completion of $\ovk$ and write ${\so}_C$ the associated ring of integers with maximal ideal ${\frak m}_C$. We write $F$ for the fraction field of the ring of Witt vectors $W({\bf F})$, $\breve{F} :=W(\overline{{\bf F}})[\frac{1}
{p}]$ for the completion of the maximal unramified extension of $F$ and $\varphi$ the Frobenius morphism on $W(\overline{{\bf F}})$. We write $\G_K:= \mathrm{Gal}(\overline{K}/K)$ for the absolute Galois group of $K$. For $V$ a $\G_K$-representation, we denote by $V(i)$ the representation obtained by twisting the action of $\G_K$ by the $i$-th power of the cyclotomic character. Let $\widehat{\B}_{\st}$ be the (completed) semistable period ring of Fontaine. It comes equipped with a Frobenius that we will again denote by $\varphi$ and a monodromy operator $N$. 


If $\Lambda$ is a (graded) ring, we will write $\Lambda-{\rm mod}$ for the category of $\Lambda$-modules and $\Lambda-{\rm grad}$ for the category of graded $\Lambda$-modules. 

All rigid analytic spaces considered will be over $K$ or $C$. 
We assume that they are separated, taut (that is, for  all quasi-compact  opens $V$ of  $X$,  the closure $\overline{V}$ of  $V$ in  $X$ is quasi-compact), and countable at infinity. We say that a rigid analytic space $X$ is Stein if it has an admissible affinoid covering $X = \bigcup_{i \in \N}U_i$ such that $U_i \Subset U_{i+1}$,
i.e., the inclusion $U_i \subset U_{i+1}$ factors through the adic compactification of $U_i$. Recall that, for such spaces, coherent sheaves are acyclic by Theorem $B$ of Kiehl (see \cite[Satz 2.4]{kie}). 

We introduce here some classical rigid spaces over $C$. We will denote by $\B^d(r)$  the closed ball of radius $r$ and dimension $d$ (and $\B^d$ when $r=1$ or $\B(r)$ when $d=0$) – not to be confused with the ball $B(n)$ for the combinatorial distance in the Bruhat-Tits building introduced in \textsection\ref{flat-prelim}, by $\A^d(r,s)$ the annulus of biradii $r\le s$, by ${\bb A}^d$ the affine space, ${\bb G}_m^d$ the torus and ${\bb P}^d$ the projective space of dimension $d$. Finally, $\mathbb{H}^d_K$ will be the Drinfeld symmetric space of dimension $d$ over $K$ and $\mathbb{H}^d:=\mathbb{H}_K^d \times_K C$ its base change to $C$.

For $M$ a $\Z$-module and $n$ an integer, we will denote by $M[n]$ its $n$-torsion, by $M[n^\infty]=\limi_{k\ge  0} M[n^k]$ its $n^\infty$ torsion and by $M[n']=\limi_{m:m\wedge n=1} M[m]$ its prime-to-$n$ torsion. 
\section{Filtrations on $\hhh^i_{\eet}(X, {\bb G}_m)$}
In this section we define two filtrations on the cohomology groups $\hhh^i_{\eet}(X, {\bb G}_m)$ for $i \ge 2$. The first one is obtained by considering the subgroup of principal units of ${\bb G}_m$ while the second one is induced by the Hodge-Tate logarithm map introduced by Heuer in \cite{Heu22}.  

\subsection{The $U$-filtration}
We begin by introducing some constructions on subsheaves of ${\bb G}_m$ from which we will derive natural filtrations on $\hhh^i_{\eet}(X, {\bb G}_m)$ by d\'evissage. 

\begin{notation}
Denote by $\nu : X_{\proeet} \to X_{\eet}$ the canonical projection of sites and let  $\tau\in\{{\eet}, \proeet\}$, consider (omitting $\tau$ in the notation when the situation is clear):
\begin{itemize}
    \item $\so_{\tau}$ the (completed) structure sheaf,
    \item $\so^{+}_{\tau}$  the sheaf of integral elements,
    \item $\so^{\times}_{\tau}$ or ${\bb G}_{m,\tau}$ the sheaf of invertible functions,
    \item $U_{\tau}:=1+ \mathfrak{m}_{C}\so^+_{\tau} \subset \so_{\tau}^{\times}$  the sheaf of principal units,
    \item $\overline{\so}^{\times}_\tau$ the quotient of $\so_{\tau}^{\times}$ by $U_{\tau}$,
    \item The $p$-adic exponential and logarithm maps define morphisms of sheaves $${\rm exp}: q\so^{+} \to 1+ q\so^{+} \text{ and } {\rm log}:1+ \mathfrak{m}_C \so^{+} \to \so$$ mutual inverse between $1 +q\so^+$ and  {$q\so^{+}$} and  where $q=p$ if $p>2$ and $q=4$ if $p=2$.
\end{itemize}
\end{notation}
The study of the long exact sequence associated to on $\hhh^i_{\eet}(X, {\bb G}_{m})$ 
\[1\to U_{\eet}\to{\bb G}_{m}\to \overline{\so}^{\times}_{\eet}\to 1\] yields a first filtration \[\FC_{U}^{(i)}:\hhh^i_{\eet}(X, {\bb G}_{m})= \FC_{U,1}^{(i)} \supset\FC_{U,0}^{(i)} \supset\FC_{U,-1}^{(i)}=0\] with graded pieces ${\rm gr}^0\FC_{U}^{(i)}\simeq\coker(\delta_{i-1})$ and ${\rm gr}^1\FC_{U}^{(i)}\simeq\ker(\delta_i)$ where $\delta_i$ is the edge morphism $\hhh^{i}_{\eet}(X, \overline{\so}^{\times})\to \hhh^{i+1}_{\eet}(X, U)$. Our goal here is to describe these terms and we will need the following facts that were, for example, proven by Heuer in  \cite{Heu2,Heu22}:

  \begin{lemma}\label{lem-heuer}
 Let $i\ge 1$, $X$ be smooth, $\nu : X_{\proeet} \to X_{\eet}$ the canonical projection of sites and $\tau\in\{{\eet}, \proeet\}$ as before. Then,
 
 \begin{enumerate}
     \item \cite[Lemma\,2.16]{Heu22} the sheaf $\overline{\so}^{\times}$ is uniquely $p$-divisible, i.e.\footnote{For an abelian sheaf $\mathcal{G}$ on $X_\tau$, write $\mathcal{G}[\tfrac{1}{p}]$ for the sheaf of abelian groups $\varinjlim_{x \mapsto x^p} \mathcal{G}$. In  particular, if $X$ is quasi-compact then for all $i$, $\hhh_{\tau}^i(X,\mathcal{G}[\tfrac{1}{p}])= \hhh_{\tau}^i(X,\mathcal{G})[\tfrac{1}{p}]$. }  $\overline{\so}_{\tau}^{\times}[\tfrac{1}{p}]\simeq \overline{\so}_{\tau}^{\times}$, \label{pdiv}
     \item \cite[Corollary\,2.11]{Heu2} we have $\R\nu_*\overline{\so}^{\times}=\overline{\so}^{\times}$ and in particular, $\hhh^i_{\eet}(X, \overline{\so}^{\times})\simeq \hhh^i_{\proeet}(X, \overline{\so}^{\times})$ for all $i \ge 0$,
     \item \cite[Lemma~2.18]{Heu22} there are exact sequences of sheaves on $X_{\tau}$ (here $\mu_{p^{\infty}} \subset U_{\tau}$ denotes the subsheaf of $p$-power unit roots):
 \begin{align}\label{even1}
& 1\to \mu_{p^{\infty}} \to U_{\tau} \xrightarrow{{\rm log}}  \so_{\tau} \to 1, \\ 
& 1\to \so_{\tau} \xrightarrow{{\rm exp}} \so_{\tau}^{\times}[\tfrac{1}{p}] \to {\overline{\so}}_{\tau}^{\times} \to 1, \notag
  \end{align} 
     \item \cite[Proposition~2.21]{Heu22} by the invertibility relations between $ {\rm log}$ and ${\rm exp}$, the  maps \eqref{even1}  induce natural isomorphisms in degree $i \ge 1$ \label{log}
 \begin{align*}
 {\rm log}: \R^i\nu_* U \xrightarrow{\sim} \R^i\nu_*\so \quad \text{ and } \quad {\rm exp}: \R^i\nu_*\so \xrightarrow{\sim} \R^i\nu_*\so^{\times}. 
 \end{align*}
 \end{enumerate}
  \end{lemma}

  We have the following key observation on affinoids: 
  \begin{lemma}
  \label{U-filtration1}
  Let $X$ be a smooth affinoid over $C$ such that $\hhh^i_{\eet}(X, \overline{\so}^{\times})$ is torsion for $i \ge 2$, then for $i\ge 2$, the groups $ \hhh^i_{\eet}(X,U)$ are $p^\infty$-torsion and $\hhh^i_{\eet}(X, \so^{\times})$ are torsion. Moreover, there are canonical isomorphisms for $i\ge2$: 
\[ \hhh^i_{\eet}(X, \so^{\times})[p] \xleftarrow{\sim}  \alpha \hhh^i_{\eet}(X, U) \text{ and }  \hhh^i_{\eet}(X, \so^{\times})[p'] \xrightarrow{\sim}\hhh^i_{\eet}(X, \overline{\so}^{\times}) \]
(with $\alpha : \hhh^i_{\eet}(X, U) \to \hhh^i_{\eet}(X, \so^{\times})$ the natural map)  and the sequence: 
  \begin{equation}
  \label{U-exact}
    0 \to \alpha \hhh^i_{\eet}(X, U) \to \hhh^i_{\eet}(X, \so^{\times}) \to  \hhh^i_{\eet}(X, \overline{\so}^{\times}) \to 0
  \end{equation}
 is exact and canonically split with $\alpha \hhh^i_{\eet}(X, U)= \hhh^i_{\eet}(X, U)$ for $i\ge 3$ generally and $i\ge 2$ when ${\rm Pic}_{\eet}(X)$ is $n$-divisible for all $n$ prime to $p$. 
  \end{lemma}
  
  \begin{proof}
 Since the coherent cohomology of $X$ is concentrated in degree $0$, the logarithm exact sequence~\eqref{even1} induces isomorphisms for $i \ge 3$,  
 \[ \hhh^{i}_{\eet}(X, U) \xleftarrow{\sim} \hhh^{i}_{\eet}(X, \mu_{p^{\infty}}) (\simeq \limi_n \hhh^{i}_{\eet}(X, \mu_{p^{n}}))\] 
(we have used here that $X$ quasi-compact; when $i=2$, it is only a surjection) and we see that $\hhh^{i}_{\eet}(X, U)$ is $p^\infty$-torsion. As $\hhh^i_{\eet}(X, {\bb G}_m)/\alpha \hhh^i_{\eet}(X, U)$ injects into $\hhh^i_{\eet}(X, \overline{\so}^{\times})$, the group  $ \hhh^i_{\eet}(X, {\bb G}_m)$ is an extension of two torsion groups and is then torsion as well.
 Moreover,  the multiplication by $p$ on the sheaf $\overline{\so}^{\times}$ is a bijection by Lemma~\ref{lem-heuer}\eqref{pdiv}, the same is true for the cohomology groups $\hhh^i_{\eet}(X, \overline{\so}^{\times})$ which is then $p'$-torsion. 
 
 Since the multiplication by $n$ is invertible on $U$ and also on $ \hhh^i_{\eet}(X, U)$ for $n$ prime to $p$,  the edge morphism  $\delta_i:\hhh^{i}_{\eet}(X, \overline{\so}^{\times})\to \hhh^{i+1}_{\eet}(X, U)$ vanishes for $i\ge 2$ and we deduce the exactness of \eqref{U-exact} and the equality $\alpha \hhh^i_{\eet}(X, U)= \hhh^i_{\eet}(X, U)$ for $i\ge 3$ generally.  Supposing the $n$-divisibility for all $n$ prime to $p$ of ${\rm Pic}_{\eet}(X)$ allows the use of the previous argument even when $i=1$ justifying  the last part of the statement since $\hhh^{i}_{\eet}(X, \overline{\so}^{\times})\simeq {\rm Pic}_{\eet}(X)[1/p]$ by Lemma~\ref{lem-heuer}\eqref{pdiv} and the exp exact sequence.
 
 Again by $p'$-torsion of $\hhh^i_{\eet}(X, \overline{\so}^{\times})$, the $p^\infty$-torsion  of $\hhh^i_{\eet}(X, \so^{\times})$ coincides with $\alpha \hhh^{i}_{\eet}(X, U)$ and  \[\hhh^i_{\eet}(X, {\bb G}_m)[p'] \simeq \hhh^i_{\eet}(X, {\bb G}_m)/\hhh^i_{\eet}(X, U) \simeq  \hhh^i_{\eet}(X, \overline{\so}^{\times}).\] Taking the inverse of this map gives the 
 desired splitting of \eqref{U-exact}.
 
  \end{proof}
 
\begin{remark}\label{ptors}
By analogy with the special case treated above, we will call, for a general smooth space $X/C$, the $p$-part of $\hhh^i_{\eet}(X, {\bb G}_m)$ the subgroup $\alpha \hhh^i_{\eet}(X, U)$, and the $p'$-part  the quotient $\im( \hhh^i_{\eet}(X, {\bb G}_m) \to  \hhh^i_{\eet}(X, \overline{\so}^{\times}))\simeq \hhh^i_{\eet}(X, {\bb G}_m)/\alpha \hhh^i_{\eet}(X, U)$ for $i\ge 2$. This choice is even further justified by the link between the torsion of $\hhh^i_{\eet}(X, {\bb G}_m)$ and the Kummer exact sequence together with its following generalization: let $n=p^km$ with $m\wedge p=1$, $X/C$ a smooth space, we have a commutative diagram of \'etale sheaves with exact rows and columns \[\begin{tikzcd}
            & 0 \arrow[d]                   & 0 \arrow[d]                        & 0 \arrow[d]                                      &   \\
0 \arrow[r] & \mu_{p^k} \arrow[r] \arrow[d] & \mu_n \arrow[r] \arrow[d]          & \mu_m \arrow[r] \arrow[d]                        & 0 \\
0 \arrow[r] & U \arrow[r] \arrow[d, "\times n"]    & {\bb G}_m \arrow[r] \arrow[d, "\times n"] & \overline{\so}^{\times} \arrow[r] \arrow[d, "\times n"] & 0 \\
0 \arrow[r] & U \arrow[r] \arrow[d]         & {\bb G}_m \arrow[r] \arrow[d]      & \overline{\so}^{\times} \arrow[r] \arrow[d]      & 0 \\
            & 0                             & 0                                  & 0                                                &  
\end{tikzcd}\]
\end{remark}

  In particular, passing to a Stein space, we obtain: 
  
  \begin{corollary}
   \label{U-filtration0}
  Let $X$ be a smooth Stein space over $C$ and assume there exists a Stein covering $\{X_n\}_{n \in \N}$ of $X$ by smooth affinoids $X_n$ such that the groups $\hhh^i_{\eet}(X_n, \overline{\so}^{\times})$ are torsion for $i \ge 2$, then we have exact sequences for all $i \ge 2$:
\[  0\to \alpha \hhh^i_{\eet}(X, U) \to \hhh^i_{\eet}(X, {\bb G}_m) \to  \hhh^i_{\eet}(X, \overline{\so}^{\times}) \to 0 \] with $\alpha \hhh^i_{\eet}(X, U)= \hhh^i_{\eet}(X, U)$ for $i\ge 3$ generally and for $i\ge 2$ when ${\rm Pic}_{\eet}(X_n)$ is $m$-divisible for all $n$ and all $m$ prime to $p$ and ${\rm R}^1\varprojlim_n\overline{\so}^{\times}(X_n)=0$. 
  

  \end{corollary}
  
  \begin{proof}
To get the desired exact sequence we need to show that the map $\hhh^i_{\eet}(X, {\bb G}_m) \to  \hhh^i_{\eet}(X, \overline{\so}^{\times})$ is a surjection for $i\ge 2$. By the usual derived formalism of the limit, note that we have the following commutative diagram with exact rows:
\[\begin{tikzcd}
0 \arrow[r] & {{\rm R}^1\varprojlim_n \hhh^{i-1}_{\eet}(X_n,{\bb G}_m)} \arrow[r] \arrow[d] & {\hhh^{i}_{\eet}(X, {\bb G}_m)} \arrow[r] \arrow[d] & {\varprojlim_n \hhh^{i}_{\eet}(X_n, {\bb G}_m)} \arrow[r] \arrow[d] & 0 \\
0 \arrow[r] & {{\rm R}^1\varprojlim_n \hhh^{i-1}_{\eet}(X_n, \overline{\so}^{\times})} \arrow[r] & {\hhh^{i}_{\eet}(X, \overline{\so}^{\times})} \arrow[r]  & {\varprojlim_n \hhh^{i}_{\eet}(X_n, \overline{\so}^{\times})} \arrow[r]  & 0
\end{tikzcd}\]
It is sufficient to see that both extremal maps are surjective. For the leftmost, ${\rm R}^1\varprojlim_n$ is right exact from the vanishing of the higher derived functor (in degree $\ge 2$) of the limit (see \cite[p13-14]{jen}). For the rightmost, we have a canonical splitting on $\hhh^i_{\eet}(X_n, {\bb G}_m) \to  \hhh^i_{\eet}(X_n, \overline{\so}^{\times})$ which then gives a splitting between the projective systems $(\hhh^i_{\eet}(X_n, {\bb G}_m))_n$ and $(\hhh^i_{\eet}(X_n, \overline{\so}^{\times}))_n$ and also between the limits, establishing the surjectivity.  By the usual long exact sequence we also get  $\alpha \hhh^i_{\eet}(X, U)= \hhh^i_{\eet}(X, U)$ for $i\ge 3$.

Now suppose that ${\rm Pic}_{\eet}(X_n)$ is $m$-divisible for all $n$ and $m$ prime to $p$, then the natural map $\varprojlim_n \hhh^2_{\eet}(X_n, U)\to \varprojlim_n \hhh^2_{\eet}(X_n, {\bb G}_m)$ is injective by Lemma \ref{U-filtration1} and  left exactness of $\limp_n$. By considering a similar diagram as above, it is  sufficient to check that ${\rm R}^1\varprojlim_n \hhh^{1}_{\eet}(X_n, U)$ injects into ${\rm R}^1\varprojlim_n \hhh^{1}_{\eet}(X_n, {\bb G}_m)$. Consider the projective system $(I_n)_n:=({\rm Im}(\overline{\so}^{\times}(X_n)\to \hhh^{1}_{\eet}(X_n, U))_n $, the following sequence is exact in the middle \[{\rm R}^1\varprojlim_n I_n\to {\rm R}^1\varprojlim_n \hhh^{1}_{\eet}(X_n, U)\to {\rm R}^1\varprojlim_n \hhh^{1}_{\eet}(X_n, {\bb G}_m).\] But the map ${\rm R}^1\varprojlim_n \overline{\so}^{\times}(X_n) \to {\rm R}^1\varprojlim_n I_n$ is surjective by right exactness of ${\rm R}^1\varprojlim_n$ and we get the desired injectivity when ${\rm R}^1\varprojlim_n\overline{\so}^{\times}(X_n)=0$.


\end{proof}

\subsection{The ${\rm HTlog}$-filtration}

We define the Hodge-Tate logarithm as the composition of the exponential map from Lemma~\ref{lem-heuer}\eqref{log} with the Hodge-Tate isomorphism proved by Scholze in \cite[Corollary~6.19]{Scholze2}, \cite[Proposition~3.23]{Scholze2b}. More precisely, recall that we have: 

\begin{proposition}[Hodge-Tate morphism]
\label{omega}
 Let $X$ be a smooth rigid space over $C$. For $\nu: X_{\proeet} \to X_{\eet}$ the natural projection, we have an $\so_X$-linear isomorphism: 
\[ {\rm HT}: \R^i\nu_*\so\xrightarrow{\sim}\Omega^i_X(-i). \]
\end{proposition} 

We shall write ${\rm HTlog}_i$ and ${\rm HTlog}_{U,i}$ the compositions 
\begin{align}
&{\rm HTlog}_i: \hhh^i_{\proeet}(X,  \so^{\times}) \to \hhh^0_{\eet}(X, \R^i\nu_{*} \so^{\times}) \xleftarrow[\sim]{{\rm exp}} \hhh^0_{\eet}(X, \R^i\nu_{*} \so) \xrightarrow{\sim} \hhh^0_{\eet}(X, \Omega_X^i(-i)), \label{HTlog}
\\
&{\rm HTlog}_{U,i}: \hhh^i_{\proeet}(X, U) \to \hhh^0_{\eet}(X, \R^i\nu_{*} U) \xrightarrow[\sim]{{\rm log}} \hhh^0_{\eet}(X, \R^i\nu_{*} \so) \xrightarrow{\sim} \hhh^0_{\eet}(X, \Omega_X^i(-i)), \label{HTlogU}
\end{align}
where the first maps in each one of the above equations correspond to the edge morphisms of the (compatible) Leray spectral sequences for the projection $\nu_*$: 
\begin{align} 
\label{leray-nu}
&E_2^{i,j}(\so^{\times})= \hhh^i_{\eet}(X, \R^j\nu_{*} \so^{\times}) \Rightarrow \hhh^{i+j}_{\proeet}(X, \so^{\times})  \\
&E_2^{i,j}(U)= \hhh^i_{\eet}(X, \R^j\nu_{*} U) \Rightarrow \hhh^{i+j}_{\proeet}(X, U). \notag
\end{align}

Note that ${\rm HTlog}_{U,i}$ is equal to the composition of the natural morphism $\hhh^i_{\proeet}(X, U) \to \hhh^i_{\proeet}(X, \so^{\times})$ with ${\rm HTlog}_i$.
We will see those two maps as edge morphisms in a certain long exact sequence.  Similarly to the previous section, this allows us to define  (compatible) filtrations $\FC_{\rm HTlog}^{(i)}$ (resp. $\FC_{\rm HTlog_U}^{(i)}$) on $\hhh^i_{\eet}(X,  \so^{\times})$ (resp. $\hhh^i_{\eet}(X, U)$) where the graded pieces are  $\coker({\rm HTlog}_i)$ and $\ker({\rm HTlog}_i)$ (resp. $\coker({\rm HTlog}_{U,i})$ and $\ker({\rm HTlog}_{U,i})$). 

Note also that for $X/C$ smooth space with trivial coherent cohomology (for example $X$ is affinoid or $X$ is Stein), degeneracy of the Leray spectral sequence for $\nu_*$ induces an isomorphism 
${\rm HT}: \hhh^i_{\proeet} (X, \so) \xrightarrow{\sim} \hhh^0_{\eet}(X, \Omega^i_X(-i))$. Moreover, by Proposition~\ref{omega} most of the terms of the exact sequence \eqref{leray-nu} vanish. Hence, the only non-zero terms in the $E_2$-page of the spectral sequences are in the row $j=0$ and column $i=0$ (for $0\le j\le d$). This gives (see~\cite[Proposition~2.13, Corollary~2.15]{EGN24} for a detailed proof): 
\begin{proposition}
\label{filtration-HT}
Let $X$ be a smooth space of dimension $d$ over $C$ with trivial coherent cohomology. Then for any $i \ge 1$, we have long exact sequences 
\begin{align*} \dots \xrightarrow{\rm HTlog_{i-1}} \Omega^{i-1}(X)(1-i)\to \hhh^i_{\eet}(X, \so^{\times}) \xrightarrow{\nu_i^*} \hhh^i_{\proeet}(X, \so^{\times}) \xrightarrow{\rm HTlog_i} \Omega^i(X)(-i)\to \cdots, \\
\cdots \xrightarrow{\rm HTlog_{U,i-1}} \Omega^{i-1}(X)(1-i) \to \hhh^i_{\eet}(X, U) \xrightarrow{\nu_{U,i}^*} \hhh^i_{\proeet}(X, U) \xrightarrow{\rm HTlog_{U,i}} \Omega^i(X)(-i)\to \cdots
\end{align*}

\end{proposition}
In particular, if $i \ge 2$, we obtain filtrations $\FC_{\rm HTlog}^{(i)}$ of $\hhh^i_{\eet}(X, \so^{\times})$ and $\FC_{{\rm HTlog}_U}^{(i)}$ of $\hhh^i_{\eet}(X, U)$: 
\begin{align}\label{kwak1}
\FC_{{\rm HTlog}}^{(i)}:\hhh^i_{\eet}(X, {\bb G}_{m})= \FC_{{\rm HTlog},1}^{(i)} \supset\FC_{{\rm HTlog},0}^{(i)} \supset\FC_{{\rm HTlog},-1}^{(i)}=0, \\  
\FC_{{\rm HTlog}_U}^{(i)}:\hhh^i_{\eet}(X, U)= \FC_{{\rm HTlog}_U,1}^{(i)} \supset\FC_{{\rm HTlog}_U,0}^{(i)} \supset\FC_{{\rm HTlog}_U,-1}^{(i)}=0, \notag 
 \end{align}
with graded pieces  ${\rm gr}^0\FC_{{\rm HTlog}}^{(i)}\simeq\coker({\rm HTlog}_{i-1})$ (resp. ${\rm gr}^0\FC_{{\rm HTlog}_U}^{(i)}\simeq\coker({\rm HTlog}_{U,i-1})$) and ${\rm gr}^1\FC_{{\rm HTlog}}^{(i)}\simeq\ker({\rm HTlog}_{i})$ (resp.  ${\rm gr}^1\FC_{{\rm HTlog}_U}^{(i)}\simeq\ker({\rm HTlog}_{U,i})$)

\begin{remark}\label{fufht}
Note that the map $\Omega^{i-1}(X)(1-i)\to \hhh^i_{\eet}(X, \so^{\times})$  is the composition $\Omega^{i-1}(X)(1-i)\to \hhh^i_{\eet}(X, U)\to \hhh^i_{\eet}(X, \so^{\times})$ so that we have an inclusion of its  image $\coker({\rm HTlog}_i)=\FC_{{\rm HTlog},0}^{(i)}$ inside $\FC_{U,0}^{(i)}$ and it is also a quotient of $\FC_{{\rm HTlog}_U,0}^{(i)}$.
\end{remark}

\subsection{Main theorem} \label{section-main-thm} We now combine the previous filtrations $\FC_{U}^{(i)}$ (see Corollary \ref{U-filtration0}) and $\FC_{\rm HTlog}^{(i)}$ (see Equation \eqref{kwak1})  to define a filtration $\FC_{\eet,{\bb G}_m}^{(i)}(X)$ (using also Remark \ref{fufht}) on the cohomology groups $\hhh^i_{\eet}(X, {\bb G}_m)$ (the $(X)$ in $\FC_{\eet,{\bb G}_m}^{(i)}(X)$ may be omitted when the situation is clear) as follows:  \[\hhh^i_{\eet}(X, {\bb G}_m)=\FC_{\eet,{\bb G}_m,2}^{(i)}\supset\FC_{\eet,{\bb G}_m,1}^{(i)}=\FC_{U,0}^{(i)}\supset\FC_{\eet,{\bb G}_m,0}^{(i)}={\FC}_{{\rm HTlog_U},0}^{(i)}\supset\FC_{\eet,{\bb G}_m,-1}^{(i)}=0,\]
and the main result will describe its graded pieces.  
This filtration can be vizualized in the following way: 
\[{\footnotesize \xymatrix{ & & 0 \ar[d] &  & \\
0 \ar[r] & \FC_{\eet,{\bb G}_m,0}^{(i)}= \alpha{\rm Coker}({\rm HTlog}_{i-1,U}) \ar[r] \ar[d] &\FC_{\eet,{\bb G}_m,1}^{(i)}=\alpha \hhh^i_{\eet}(X,U) \ar[r]^-{\nu_{U,i}^*} \ar[d] & {\rm gr}^1\FC_{\eet,{\bb G}_m}^{(i)}= \iota{\rm Ker}({\rm HTlog}_{i,U}) \ar[r] \ar[d] & 0 \\
0 \ar[r] & {\rm Coker}({\rm HTlog}_{i-1})  \ar[r] & \FC_{\eet,{\bb G}_m,2}^{(i)}=\hhh^i_{\eet}(X, \so^{\times}) \ar[d] \ar[r]^-{\nu_i^*} &{\rm Ker}({\rm HTlog}_{i}) \ar[r] &  0 \\
& & {\rm gr}^2\FC_{\eet,{\bb G}_m}^{(i)}(\subset \hhh^i_{\eet}(X,\overline{\so}^{\times}))  \ar[d]  &  &  \\
& & 0  &   &}} \] 

\begin{remark}
By the long exact sequence in Proposition \ref{filtration-HT}, we can define, by a similar argument, a filtration $\FC_{\proeet,{\bb G}_m}^{(i)}$ of $\hhh^i_{\proeet}(X, \so^{\times})$   (using that $\FC_{\eet,{\bb G}_m,0}^{(i)}=\ker(\hhh^i_{\eet}(X, \so^{\times}) \to \hhh^i_{\proeet}(X, \so^{\times}))$): \[\hhh^i_{\proeet}(X, {\bb G}_m)=\FC_{\proeet,{\bb G}_m,2}^{(i)}\supset\FC_{\eet,{\bb G}_m,2}^{(i)}/\FC_{\eet,{\bb G}_m,0}^{(i)}\supset\FC_{\eet,{\bb G}_m,1}^{(i)}/\FC_{\eet,{\bb G}_m,0}^{(i)}\supset 0,\] where ${\rm gr}^2\FC_{\proeet,{\bb G}_m}$ is $\ker(\Omega^{i-1}(X)\to \FC_{\eet,{\bb G}_m,1}^{(i)})={\rm Im}({\rm HTlog}_{i})$.
\end{remark}

For $X$ a smooth rigid analytic space over $C$, let us recall the definition of the $i$-th $\wh{\B}^+_{\st}$-twisted Hyodo-Kato cohomology group from Colmez-Dospinescu-Nizio{\l} in \cite{CDN20}. It is defined by the formula
\[ {\rm HK}^i(X):= (\hhh^{i}_{{\rm HK}}(X){\otimes}^{\Box}_{\breve{F}}\wh{\B}^+_{\st})^{N=0,\phi=p^{i}}, \]
where $N$ denotes the monodromy operator and $\phi$ the Frobenius morphism. See \cite[\textsection 2]{CN25a} for the definition of the Hyodo-Kato cohomology groups $\hhh^{i}_{{\rm HK}}(X)$ and some more properties of ${\rm HK}^i(X)$. Here the tensor product is taken in the category of (solid) $\breve{F}$-vector spaces. Then ${\rm HK}^i(X)$ is a (solid) $\Q_p$-vector spaces, which can be seen as the $p$-adic analog of $\ell$-adic \'etale cohomology. It is endowed with an action of the absolute Galois group $\G_{K}$ of $K$ and we write ${\rm HK}^i(X)(j)$ for its twist by the $j$-th power the cyclotomic character of $\G_{K}$. 

Consider the following properties for $(X,(X_n)_n))$ a Stein space together with a fixed Stein covering $X=\bigcup_n X_n$:

 ${\bf (H_1')}$ (resp. ${\bf (H_1)}$) The groups $\hhh^i_{\eet}(X_n, \overline{\so}^{\times})$ are torsion for $i \ge 2$ (resp. $i\ge 1$ and ${\rm R}^1\varprojlim_n \overline{\so}^{\times}(X_n)=0$).

${\bf (H_2')}$ (resp. ${\bf (H_2)}$) For $i\geq 2$ (resp. $i\geq 1$), there exists a projective system $(N^{(i)}_n)_{n \in \N}$ of finite free $\Z[\frac{1}{p}]$-modules such that for any prime $\ell\neq p$ and $n\ge 1$, there is an isomorphism of projective systems  $(N^{(i)}_n\otimes \Z/\ell^m\Z)_n \xrightarrow{\sim}  (\hhh^i_{\eet}(X_n, \mu_{\ell^m}))_n$ (resp. and the group ${\rm Pic}_{\eet}(X_n)$ is $\ell$-divisible for any prime $\ell \neq p$). 

${\bf (H_3)}$ The map $\hhh^i_{\eet}(X, \Z_p(i)) \to {\rm HK}^i(X)$ is injective.

Our main result is the following:

\begin{theorem}
\label{main-theorem}
 Let $(X,(X_n)_n))$ be a smooth Stein space over $C$  with a fixed Stein covering and let $i\ge 3$ (resp. $i\ge 2$). If $(X,(X_n)_n))$ satisfy the property ${\bf (H_1')}$ (resp.  ${\bf (H_1)}$), then \[\FC_{\eet,{\bb G}_m,1}^{(i)}= \hhh^i_{\eet}(X,U) \text{ and } {\rm gr}^2\FC_{\eet,{\bb G}_m}^{(i)}=\hhh^i_{\eet}(X,\overline{\so}^{\times}).\] 
 If moreover $(X,(X_n)_n))$ satisfies ${\bf (H_2')}$ (resp. ${\bf (H_2)}$), then 
 $${\rm gr}^2\FC_{\eet,{\bb G}_m}^{(i)}\simeq \widehat{\bigoplus}_{\ell \neq p} \hhh^i_{\eet}(X, \mu_{\ell^{\infty}}),$$
 where the sum is completed for some well-choosen topology\footnote{See \textsection~\ref{l-adic} for the definition of this topology.}.

Finally, if $(X,(X_n)_n))$ satisfy ${\bf (H_3)}$ as well, 
\begin{align*}
   &{\rm gr}^0 \FC_{\eet,{\bb G}_m}^{(i)} \simeq \Omega^{i-1}(X)(1-i)/{\rm Ker}(d) \\
   &{\rm gr}^1 \FC_{\eet,{\bb G}_m}^{(i)} \simeq {\rm HK}^i(X)(1-i)/ \hhh^{i}_{\proeet}(X, \Z_p(1)). 
\end{align*}
\end{theorem}

Before explaining the proof of the theorem, let us make the following remark. 
\begin{remark}
\begin{itemize}
    \item It is expected that the torsion hypothesis of  ${\bf (H_1')}$ is satisfied by a big enough class of smooth affinoids and should be seen as the main hypothesis of the theorem.
    \item  ${\bf (H_2')}$ is a relatively natural hypothesis to simplify the statement of the main theorem.
    \item The extra hypothesis imposed in  ${\bf (H_1)}$ an  ${\bf (H_2)}$ are necessary technicalities to avoid some pathologies coming from the limit argument. They should fail for reasonable family of smooth affinoids. Fortunately for us, in all the examples we are considering, they can be checked directly because the Picard group is trivial and the sections $\overline{\so}^{\times}(X_n)$ are already determined in the literature.
    \item Condition  ${\bf (H_3)}$ is here to avoid extra pathologies of the $p$-adic cohomology of affinoid spaces which are inherited by unbounded Stein spaces (for example the open unit ball does not satisfy  ${\bf (H_3)}$ whereas it is the case for the affine space). Without this hypothesis we can still define the filtration and describe its graded pieces but in a bit more complicated and less explicit manner. More precisely, we have another filtration $ \tilde{\FC}_{\eet,{\bb G}_m}^{(i)}$ : \[\hhh^i_{\eet}(X, {\bb G}_m)=\tilde{\FC}_{\eet,{\bb G}_m,3}^{(i)}\supset\tilde{\FC}_{\eet,{\bb G}_m,2}^{(i)}\supset\tilde{\FC}_{\eet,{\bb G}_m,1}^{(i)}\supset\FC_{\eet,{\bb G}_m,0}^{(i)}\supset\tilde{\FC}_{\eet,{\bb G}_m,-1}^{(i)}=0,\] where, under  ${\bf (H_1)}$ and ${\bf (H_2)}$, the graded pieces are
    \begin{align*}
    &{\rm gr}^0 \tilde{\FC}_{\eet,{\bb G}_m}^{(i)} \simeq \Omega^{i-1}(X)(1-i)/{\rm Exp}_{i-1}^{-1}({\rm Im}(\iota)) \\
    &{\rm gr}^1 \tilde{\FC}_{\eet,{\bb G}_m}^{(i)} \simeq {\rm HK}^i(X)(1-i)/ \iota(\hhh^{i}_{\proeet}(X, \Z_p(1)))\\
    &{\rm gr}^2 \tilde{\FC}_{\eet,{\bb G}_m}^{(i)} \simeq {\rm Ker}(\iota) \\
    &{\rm gr}^3 \tilde{\FC}_{\eet,{\bb G}_m}^{(i)} \simeq {\rm HK}^i(X)(1-i)/ \hhh^{i}_{\proeet}(X, \Z_p(1)).
    \end{align*} where ${\rm Exp}_i : (\Omega^i(X))(-i) \to \hhh^{i+1}_{\proeet}(X, \Q_p(1))$ is the Bloch-Kato exponential and $\iota :  \hhh^{i+1}_{\proeet}(X, \Z_p(1)) \to \hhh^{i+1}_{\proeet}(X, \Q_p(1))$ the natural map. 
    
    Note that under  ${\bf (H_3)}$, one has ${\rm Ker}(\iota\circ{\rm Exp}_i)={\rm Ker}({\rm Exp}_i)={\rm Ker}(d)$ and ${\rm gr}^2 \tilde{\FC}_{\eet,{\bb G}_m}^{(i)}=0$. We recover the main theorem by observing $\tilde{\FC}_{\eet,{\bb G}_m,0}^{(i)}=\tilde{\FC}_{\eet,{\bb G}_m,0}^{(i)}$ and $\tilde{\FC}_{\eet,{\bb G}_m,1}^{(i)}=\tilde{\FC}_{\eet,{\bb G}_m,2}^{(i)}={\FC}_{\eet,{\bb G}_m,1}^{(i)}$.
\end{itemize}
\end{remark}

\begin{proof}[Proof of Theorem~\ref{main-theorem}]
The consequences of property ${\bf (H_1)}$ have been established in Lemma~\ref{U-filtration1} and Corollary~\ref{U-filtration0}. Determining the terms ${\rm gr}^0\FC_{\eet,{\bb G}_m}^{(i)}$ and  ${\rm gr}^1 \FC_{\eet,{\bb G}_m}^{(i)}$ is the goal of Section \ref{p-adic} (using the fact that the map $\alpha :\hhh^i_{\eet}(X, U) \to \hhh^i_{\eet}(X, \so^{\times})$ is injective under these hypotheses and that both $\coker({\rm HTlog}_i)$ and $\coker({\rm HTlog}_{U,i})$ coincide), where they are computed respectively in Proposition~\ref{gr0} and in \eqref{gr1}. The calculation of the $p'$-part is developed in Section \ref{l-adic} (see Corollary~\ref{gr2}).  
\end{proof}

By Proposition~\ref{filtration-HT} and Remark~\ref{fufht} there is an analogous statement for the pro-\'etale cohomology.

\begin{corollary}
Let $(X,(X_n)_n))$ be a smooth Stein space over $C$  with a fixed Stein covering which satisfies ${\bf (H_1)}$,  ${\bf (H_2)}$ and  ${\bf (H_3)}$, then for $i \ge 2 $  we have
\begin{align*}
   &{\rm gr}^0 \FC_{\proeet,{\bb G}_m}^{(i)} \simeq {\rm HK}^i(X)(1-i)/ \hhh^{i}_{\proeet}(X, \Z_p(1)) \\
   &{\rm gr}^1 \FC_{\proeet,{\bb G}_m}^{(i)} \simeq \widehat{\bigoplus_{\ell \neq p}} \hhh^i_{\eet}(X, \mu_{\ell^{\infty}}) \\
   &{\rm gr}^2 \FC_{\proeet,{\bb G}_m}^{(i)} \simeq \Omega^{i}(X)(-i)^{d=0}.
\end{align*} 
\end{corollary}

\begin{proof}
By the explicit description of the filtration $\FC_{\proeet,{\bb G}_m}^{(i)}$ in Remark~\ref{fufht}, we deduce identifications \[{\rm gr}^0 \FC_{\proeet,{\bb G}_m}^{(i)}\simeq {\rm gr}^1 \FC_{\eet,{\bb G}_m}^{(i)},\ \ {\rm gr}^1 \FC_{\proeet,{\bb G}_m}^{(i)}\simeq {\rm gr}^2 \FC_{\eet,{\bb G}_m}^{(i)}\] and ${\rm gr}^2 \FC_{\proeet,{\bb G}_m}^{(i)}\simeq \ker (\Omega^{i}(X)(-i)\to {\rm gr}^0 \FC_{\eet,{\bb G}_m}^{(i+1)})$  and all of these terms are computed by Theorem~\ref{main-theorem} under ${\bf (H_1)}$,  ${\bf (H_2)}$ and  ${\bf (H_3)}$.
\end{proof}

As an application, we can determine the cohomology of some classical Stein spaces such as the affine space ${\bb A}_C^d$, the torus ${\bb G}_{m,C}^d$ or the Drinfeld symmetric space ${\bb H}^d$ (see \textsection\ref{open-drinfeld} for a definition of ${\bb H}^d$ and a description of its standard Stein covering). In the following, for $A$ in $\{\Z_{\ell} \; | \; \ell \text{ any prime}\}$, we denote by ${\rm Sp}_i(A)$ the generalized locally constant Steinberg representation of $G:= {\rm GL}_{d+1}(K)$ (see for example \cite[Section~5.2]{CDN20} for a definition). We write ${\rm Sp}_i(A)^*$ for its dual for the weak topology.  

\begin{corollary}\label{mainex}
 For $i\ge 2$, we have isomorphisms of $\G_K$-representations: 
\[ \hhh^i_{\eet}({\bb A}_C^d, \so^{\times}) \simeq \Omega^{i-1}({\bb A}_C^d)(-i+1)/{\rm Ker}(d); \] 
\begin{align*}
{\rm gr}^0 \FC_{\eet,{\bb G}_{m}}^{(i)}({\bb G}_{m,C}^d) &\simeq \Omega^{i-1}({\bb G}^d_{m,C})(-i+1)/{\rm Ker}(d), \ \ \ {\rm gr}^1 \FC_{\eet,{\bb G}_m}^{(i)}({\bb G}_{m,C}^d) \simeq (\wedge^{i} \Q^d_p/\wedge^{i} \Z^d_p)(-i+1), \\ 
&{\rm gr}^2 \FC_{\eet,{\bb G}_m}^{(i)}({\bb G}_{m,C}^d) \simeq \bigoplus_{\ell\neq p} (\wedge^{i} \Q^d_\ell/\wedge^{i} \Z^d_\ell)(-i+1); 
\end{align*}
\begin{align*}
{\rm gr}^0 \FC_{\eet,{\bb G}_m}^{(i)}({\bb H}^d) &\simeq \Omega^{i-1}({\bb H}^d)(-i+1)/{\rm Ker}(d), \ \ \ {\rm gr}^1 \FC_{\eet,{\bb G}_m}^{(i)}({\bb H}^d) \simeq  ({\rm Sp}_{i}(\Q_{p})^*/{\rm Sp}_{i}( \Z_p)^*)(-i+1), \\
&{\rm gr}^2 \FC_{\eet,{\bb G}_m}^{(i)}({\bb H}^d) \simeq \widehat{\bigoplus_{\ell\neq p}} (({\rm Sp}_{i}(\Z_\ell)^* \otimes \Q_{\ell})/{\rm Sp}_{i}(\Z_\ell)^*)(-i+1).
\end{align*}
\end{corollary}

\begin{proof}
It is essentially due to the fact that all of these spaces together with their usual covering satisfy ${\bf (H_1)}$, ${\bf (H_2)}$ and ${\bf (H_3)}$. Let us develop a bit how to establish these facts for these spaces.  Note that ${\rm Pic}(X)={\rm Pic}(X_n)=0 $ (\cite[Theorem 3.25]{vdp} for ${\bb A}_C^d$,  ${\bb G}_{m,C}^d$ and their product, \cite[Théorème 6.7, Théorème 7.1]{J1} for ${\bb H}^d$) and $\overline{\so}^{\times}(X_{n+1})\to \overline{\so}^{\times}(X_n)$ is surjective (it is \cite[Lemme 4.4., Théorème 6.10]{J1} for the presheaf ${\bb G}_m/U$ and the result is obtained by inverting $p$  by Lemma \ref{lem-heuer}\eqref{pdiv}) for $X$ one of those spaces.   It is then enough to establish  ${\bf (H_1')}$ and ${\bf (H_2')}$ instead of  ${\bf (H_1)}$ and  ${\bf (H_2)}$. The proof of  ${\bf (H_1')}$  is the topic of Section \ref{torsion} (see Theorem~\ref{thtorsion}). 

For  ${\bf (H_2')}$, one has  for ${\bb A}_C^r\times{\bb G}_{m,C}^s$ by \cite[Lemma 3.3.]{ber6} and Künneth formula \[\hhh^i_{\eet}(X_n, \Z/\ell^n\Z)\simeq \wedge^{i} (\Z/\ell^n\Z)^s(-i)\] which establishes  ${\bf (H_2')}$ for ${\bb A}_C^r\times{\bb G}_{m,C}^s$ with $N^{(i)}_j\simeq N^{(i)}_{j'}=\wedge^{i} \Z^s$ for any $j'\neq j$. Extending this property to the symmetric space is the topic of Section \ref{flat-0} (see Corollary \ref{flat-2}). The computation of the $\ell$-adic \'etale cohomology for ${\bb H}^d$ was done by Schneider and Stuhler in~\cite{scst} (see also \cite[Theorem 5.8]{CDN20}).

Now for the property  ${\bf (H_3)}$ observe the following description for all $i \ge 0$,  
\begin{align*}
    \hhh^i_{\proeet}({\bb A}_C^r\times{\bb G}_{m,C}^s, \Z_{p}) \simeq \hhh^i_{\eet}({\bb A}_C^r\times{\bb G}_{m,C}^s, \Z_{p}) \simeq \wedge^i \Z_{p}^s(-i)\ &\text{ and }\ {\rm HK}^i({\bb A}_C^r\times{\bb G}_{m,C}^s) \simeq \wedge^i \Q_{p}^s(-i) \\
   \hhh^i_{\proeet}({\bb H}^d, \Z_{p}) \simeq \hhh^i_{\eet}({\bb H}^d, \Z_{p}) \simeq {\rm Sp}_{i}(\Z_p)^*\ &\text{ and }\ {\rm HK}^i({\bb H}^d) \simeq {\rm Sp}_{i+1}(\Q_p)^*,
\end{align*}
where the first isomorphism in each rows follows from the fact that the space admits an increasing covering by affinoid spaces and that the equality is true for affinoids by~\cite[Corollary~3.17]{Scholze2}, see for example \cite[Proof of Corollary~3.46]{CDN20}. The map  $\hhh^i_{\proeet}(X, \Z_p(1)) \to {\rm HK}^i(X)(i-1)$, in all three  cases, is the natural one and it is clearly injective. These computations were done in \cite{ber3} for integral cohomology of the torus and the affine space, \cite[\textsection 2]{CN20} for the twisted Hyodo-Kato cohomology of the affine space, \cite[\textsection 4.3.2 and \textsection 5]{CDN20} for the Hyodo-Kato cohomology of the Drinfeld space and the torus and \cite{CDN21} for the integral cohomology of the Drinfeld space.

\end{proof}


\section{Image and kernel of the morphism ${\rm HTlog}_U$} \label{p-adic}
In this section we briefly recall the results from \cite{EGN24}. In particular, we compute  the $p$-part of the ${\bb G}_m$-cohomology. 

\subsection{Image of ${\rm HTlog}_{U,i}$}

For $i \ge 1$ and $X$ a smooth Stein space over $C$, the image of the Hodge-Tate logarithm ${\rm HTlog}_{U,i}$ was computed in~\cite{EGN24}. We first recall here the statement of the main theorem and briefly sketch the proof. 

\subsubsection{General formula}
The main idea to compute the image of ${\rm HTlog}_{U,i}$ is to compare it, via the Bloch-Kato exponential map, with the $p$-adic pro-\'etale cohomology. 

In this setting, the Bloch-Kato exponential map was first defined in~\cite[Proposition~3.36]{CDN20}, using comparison theorems between pro-\'etale and syntomic cohomologies. It is an injective morphism of (solid) $\Q_p$-vector spaces,
\begin{equation}
\label{CDN}
 {\rm Exp}_{i} : (\Omega^i(X)/{\rm Ker}\,d)(-i) \hookrightarrow \hhh^{i+1}_{\proeet}(X, \Q_p(1)). \end{equation}
Here, we will in fact use the alternative construction from~\cite{Bos23} where the Bloch-Kato exponential map is defined using the edge morphism coming from the fundamental exact sequence of $p$-adic Hodge theory: 
\[ 0 \to \Q_p(1) \to \mathbb{B}^{\varphi=p} \to  \mathbb{B}^+_{\rm dR}/t \simeq \so \to 0,\]
see for example \cite[\textsection 8]{LB18} for the definition of these sheaves.

The main result from~\cite{EGN24} is the following:  

\begin{theorem}
\label{vpic-main}
Let $X$ be a smooth Stein rigid analytic space over $C$ and let $i\geq 1$. Then, the image of the restriction of the Hodge-Tate logarithm to the cohomology group of principal units 
\[ {\rm HTlog}_{i,U} : \hhh^i_{\proeet}(X,U) \to \Omega^i_{X}(X)(-i) \]
fits into a short exact sequence of abelian groups
\begin{equation}
\label{image0}
0 \to  \Omega^i(X)(-i)^{d=0} \to {\rm Im}({\rm HTlog}_{i,U}) \xrightarrow{ {\rm Exp}_i}  \mathcal{I}^i(X) \to 0 
\end{equation}
where the $\Z_p$-module $\mathcal{I}^i(X) \subset \hhh^{i+1}_{\proeet}(X, \Q_p(1))$ is the intersection ${\rm Im}({\rm Exp}_i) \cap {\rm Im}(\iota)$, where
$\iota: \hhh^{i+1}_{\proeet}(X, \Z_p(1)) \to \hhh^{i+1}_{\proeet}(X, \Q_p(1))$ is the canonical map. 
\end{theorem}

Recall (\cite[Proof of Corollary~3.46]{CDN20}) that if $X$ is a Stein space with Stein covering $\{X_n\}$, we have 
\[{\rm R}\Gamma_{\proeet}(X, \Z_p) \simeq {\rm R}\Gamma_{\eet}(X, \Z_p).  \] 

Let us briefly explain how Theorem~\ref{vpic-main} is proved. We work on the pro-\'etale site of $X$ and start with the logarithmic exact sequence from~\eqref{even1}. Taking pro-\'etale cohomology, we obtain a diagram: 
\begin{equation}
\label{diagr-EGN}
 \xymatrix{ \hhh^i_{\proeet}(X, U) \ar[r]^-{{\rm HTlog}_{i,U}}  & \Omega^i(X)(-i) \ar[r]^-{\partial_{\rm log}\circ {\rm HT}^{-1}} \ar@{->>}[d] & \hhh^{i+1}_{\proeet}(X, \Q_p/\Z_p(1)) \\
& \Omega^i(X)/{\rm Ker}(d)(-i) \ar@{^(->}[r]^-{{\rm Exp}_i}  & \hhh^{i+1}_{\proeet}(X, \Q_p(1))  \ar[u] \\ & & \hhh^{i+1}_{\eet}(X, \Z_p(1)) \ar[u]^-{\iota},   }
\end{equation}
where the first row and the column on the right are exact. Here $\partial_{\rm log}$ is the edge map coming from the logarithmic exact sequence. We have used the identification $\hhh^{i+1}_{\eet}(X, \Z_p(1)) \simeq \hhh^{i+1}_{\proeet}(X, \Z_p(1))$. 
The most technical part (see~\cite[Section~3]{EGN24}) is to prove that this diagram is commutative. The exact sequence~\eqref{image0} then follows by a diagram chase.

\subsubsection{Vanishing of ${\cal I}^i(X)$} The cokernel of the Bloch-Kato exponential map can in fact be computed. For $i \ge 1$, there is an exact sequence of (solid) $\Q_p$-vector spaces (see~\cite[Theorem~1.8]{CDN20} and \cite[Theorem~7.7]{Bos23}): 
\[ 0 \to \Omega^i(X)/{\rm Ker}\,d \xrightarrow{{\rm Exp}} \hhh^{i+1}_{\proeet}(X, \Q_p(i+1)) \to  {\rm HK}^{i+1}(X):=(\hhh^{i+1}_{{\rm HK}}(X){\otimes}^{\Box}_{\breve{F}}\wh{\B}^+_{\st})^{N=0,\phi=p^{i+1}} \to 0 \]
Hence, we see that the group ${\cal I}^i(X)$ is zero when the map from $\hhh^{i+1}_{\eet}(X, \Z_p(i+1))/{T_{i+1}}$, where $T_i$ is the maximal tosion subgroup, to the Hyodo-Kato term above is injective. We obtain: 

\begin{proposition}
\label{gr0}
Let $X$ be a smooth Stein space such that the map $\hhh^{i}_{\eet}(X, \Z_p(i))/{T_i} \to {\rm HK}^{i}(X)$ is injective for all $i\ge 1$. Then, for all $i \ge 1$, 
\[ {\rm Coker}({\rm HTlog}_{i,U} )\simeq \Omega^i(X)(-i)/{\rm Ker}(d). \]
\end{proposition}


\subsection{Kernel of ${\rm HTlog}_U$}
\label{ker-HTlog}
Let us now recall the computation of the kernel of the maps ${\rm HTlog}_{U,i}, i \ge 1$ from~\cite[Section~5.2.1]{EGN24}. Consider $X$ a smooth Stein space over $C$ and assume that the map $ \hhh^{i}_{\eet}(X, \Z_p(i)) \to {\rm HK}^{i}(X)$ is injective (in particular, $\hhh^{i}_{\eet}(X, \Z_p(i))$ is torsion-free). The logarithm exact sequence induces a commutative diagram with exact rows for all $i\ge 2$: 
\begin{equation}
\xymatrix{
0 \ar[r] &(\Omega^{i-1}(X)/{\rm Ker}\,d)(1-i) \ar[r]^-{\rm Exp} \ar[d]^{\rotatebox{90}{$\sim$}} & \hhh^i_{\proeet}(X, \Q_p(1)) \ar[d]^{f_1} \ar[r] & {\rm HK}^i(X)(1-i) \ar[d]^{f_2} \ar[r] & 0 \\ 
0 \ar[r] &{\rm Coker}({\rm HTlog}_{U,{i-1}})  \ar[r] &  \hhh^i_{\proeet}(X, \Q_p/\Z_p(1)) \ar[r] & {\rm Ker}({\rm HTlog}_{U,i}) \ar[r] & 0.} \end{equation}  

By hypothesis, we also have that the map $f_1$ is surjective and so the same is true for the map $f_2$. This yields an isomorphism 
\begin{equation}
    \label{gr1}
   {\rm Ker}({\rm HTlog}_{U,i}(X)) \simeq {\rm HK}^i(X)(1-i)/\hhh^i_{\eet}(X, \Z_p(1)).
\end{equation}

\section{$p'$-part of ${\bb G}_m$-cohomology} \label{l-adic}

In this section we compute the prime-to-$p$ torsion of the ${\bb G}_m$-cohomology groups, i.e. the graded piece in degree $2$ for the filtration from Theorem~\ref{main-theorem}. We obtain the  $\ell$-adic analogue of equation~\eqref{gr1}. 

\begin{proposition}
\label{O-bar-stein0}
Let $X$ be a smooth Stein rigid space over $C$ with a Stein covering $\{X_j\}$ satisfying $({\bf H}_1)$ and $({\bf H}_2)$. Then, for all $i \ge 2$, there is an isomorphism:  
 \[ \hhh^i_{\eet}(X, \overline{\so}^{\times}) \simeq \varprojlim_j \big(\bigoplus_{\ell \neq p} \hhh^i_{\eet}(X_j,\mu_{\ell^{\infty}})\big) \simeq  \varprojlim_j \big(\bigoplus_{\ell \neq p} \hhh^i_{\eet}(X_j, \Q_{\ell}(1))/\hhh^i_{\eet}(X_j, \Z_{\ell}(1))\big). \]
\end{proposition}

It relies on the following two lemmas: 

\begin{lemma}
\label{r1lim-nul}
Assume that $X$ is a Stein space with a covering $\{ X_j \}_{j \in \N}$ satisfying $(\mathbf{H}_2)$. Then, for all $i \ge 1$, 
\[ \R^1\varprojlim_{j \in \N} \big( \bigoplus_{\ell \neq p} \hhh^i_{\eet}(X_j, \mu_{\ell^{\infty}})\big) =0. \]
\end{lemma}

\begin{proof}
Let $N^{(i)}_j$ be the projective finite free $\Z[\frac{1}{p}]$-module from hypothesis $(\mathbf{H}_2)$. By hypothesis and Chinese remainder theorem, we have compatible isomorphisms $ N^{(i)}_j\otimes_{\Z[1/p]}\Z/n \Z \simeq \hhh^i_{\eet}(X_j, \mu_{n})$ for all $n$ coprime to $p$. We obtain a surjective morphism of projective systems  
\[ (N^{(i)}_j \otimes_{\Z[\frac{1}{p}]} \Q)_j \twoheadrightarrow (\varinjlim_{n \wedge p=1} N^{(i)}_j\otimes_{\Z[1/p]}\Z/n \Z)_j \xrightarrow{\sim}  (\varinjlim_{n \wedge p=1} \hhh^i_{\eet}(X_j, \mu_{n}))_j. \]
But the modules $N^{(i)}_j$'s are finite free and $\Q$ is artinian, so the system on the left satisfies Mittag-Leffler conditions and its $\R^1\varprojlim$ is zero. This yields the desired result by right exactness of $\R^1\varprojlim$.  
\end{proof}

\begin{lemma}
\label{flat-div}
Let $X$ be an affinoid over $C$. If for all $i \ge 0$ and all $n \ge 1$, the groups $\hhh^i_{\eet}(X, \Z/{\ell^n}\Z)$ are flat over $\Z/\ell^n\Z$ then for all $i \ge 0$, the group $\hhh^i_{\eet}(X, \varinjlim_n \Z/{\ell^n}\Z)$ is divisible. 
\end{lemma}

\begin{proof}
For $n \in \N$, the exact sequence
\[0 \to \mu_{\ell^n} \to \mu_{\ell^{\infty}} \xrightarrow{\ell^n} \mu_{\ell^{\infty}} \to 0  \] 
is obtained by taking the colimit of the compatible exact sequences (for $m \ge n$): 
\[ \xymatrix{0 \ar[r] & \mu_{\ell^n} \ar@{=}[d] \ar[r] & \mu_{\ell^m} \ar[r]^-{\ell^n} \ar[d] & \mu_{\ell^{m-n}} \ar[r] \ar[d] & 0 \\ 
0 \ar[r] &\mu_{\ell^n} \ar[r] &  \mu_{\ell^{m+1}} \ar[r]^-{\ell^n} & \mu_{\ell^{m+1-n}} \ar[r] & 0} \] 
We want to prove that the multiplication by $\ell^n$ from $\hhh^i_{\eet}(X, \mu_{\ell^{\infty}})$ to itself is surjective for all $i \ge 0$. So it is enough to show that for all $m \ge n$, the natural map $\hhh^i_{\eet}(X, \mu_{\ell^n}) \to \hhh^i_{\eet}(X, \mu_{\ell^m})$ is injective for all $i \ge 0$. 

We prove the latter result by induction on the degree $i$. For $i=0$, this is clear. Suppose the result for $i$ and consider the following diagram where the first row is exact: 
\begin{equation}\label{flat-diagram2} \xymatrix{ 0 \ar[r] & \hhh^i_{\eet}(X, \mu_{\ell^n}) \ar[r] & \hhh^i_{\eet}(X, \mu_{\ell^m}) \ar[r]^{\ell^n} \ar[rd]_{\ell^n} & \hhh^i_{\eet}(X, \mu_{\ell^{m-n}}) \ar@{^(->}[d] \\ & & & \hhh^i_{\eet}(X, \mu_{\ell^n})} \end{equation} 
It follows that $\hhh^i_{\eet}(X, \mu_{\ell^n})$ identifies with the $\ell^n$-torsion subgroup of $\hhh^i_{\eet}(X, \mu_{\ell^m})$. By flatness, we have an exact sequence: 
\[0 \to \hhh^i_{\eet}(X, \mu_{\ell^m})[\ell^n] \to \hhh^i_{\eet}(X, \mu_{\ell^m}) \xrightarrow{\ell^n} \hhh^i_{\eet}(X, \mu_{\ell^m})[\ell^{m-n}] \to 0  \]
which can then be identified with the top row of diagram~\ref{flat-diagram2}. It follows that the boundary map vanishes and that $\hhh^{i+1}_{\eet}(X, \mu_{\ell^n}) \to \hhh^{i+1}_{\eet}(X, \mu_{\ell^m})$ is injective, as wanted.  
\end{proof}

\begin{proof}[Proof of Theorem \ref{O-bar-stein0}]
The second isomorphism in the statement follows from the fact that the groups $\hhh^i_{\eet}(X_j, \mu_{\ell^{\infty}})$ are divisible by Lemma~\ref{flat-div}: in that case, the $\Z_{\ell}$-modules $\hhh^i_{\eet}(X_j, \Z_{\ell})= \varprojlim_n \hhh^i_{\eet}(X_j, \mu_{\ell^n})$ are finite free and we have
\[ \hhh^i_{\eet}(X_j, \mu_{\ell^{\infty}}) \simeq \hhh^i_{\eet}(X_j, (\Q_{\ell}/\Z_{\ell})(1)) \simeq  \hhh^i_{\eet}(X_j, \Q_{\ell}(1))/\hhh^i_{\eet}(X_j,\Z_{\ell}(1)). \]


In order to prove the first isomorphism, let us fix an affinoid $X_j$ and show that for all $i \ge 1$ and for $n$ prime to $p$, the left term in the Kummer exact sequence: 
\begin{equation*}
 0 \to \hhh^{i-1}_{\eet}(X_j, \overline{\so}^{\times})/(n \cdot \hhh^{i-1}_{\eet}(X_j, \overline{\so}^{\times})) \to \hhh^i_{\eet}(X_j, \mu_n) \to \hhh^i_{\eet}(X_j, \overline{\so}^{\times})[n] \to 0   
\end{equation*}
vanishes, i.e. $\hhh^{i}_{\eet}(X_j, \overline{\so}^{\times})$ is $\ell$-divisible for any $i \ge 1$ and any prime $\ell \neq p$.  

We prove the result by induction on $i$. By assumption ${\rm Pic}_{\eet}(X_j)$ is $\ell$-divisible for $\ell \neq p$ and the exponential exact sequence induces an isomorphism
\[{\rm Pic}(X_j)[\frac{1}{p}] \xrightarrow{\sim} \hhh^1(X_j, \overline{\so}^{\times}),\]
so $\hhh^1(X_j, \overline{\so}^{\times})$ is $n$-divisible for any $n$ prime to $p$. 

Let $i \ge 2$ and assume that $\hhh^{i-1}_{\eet}(X_j, \overline{\so}^{\times})$ is $\ell$-divisible for any $\ell \neq p$. We know that $\hhh^i_{\eet}(X_j, \overline{\so}^{\times})$ is prime-to-$p$-torsion. 
By induction hypothesis, we have that the group $\hhh^i_{\eet}(X_j, \mu_n)$ is isomorphic to $\hhh^i_{\eet}(X_j, \overline{\so}^{\times})[n]$. By hypothesis, the group $\varinjlim_n \hhh^i_{\eet}(X_j, \mu_{\ell^n}) \simeq \hhh^i_{\eet}(X_j, \mu_{\ell^{\infty}})$ is $\ell$-divisible. This concludes the proof of the induction.

By Chinese Remainder Theorem, we obtain that (by using that $X_j$ is quasi-compact),
\[\hhh^i_{\eet}(X_j, \overline{\so}^{\times}) \simeq \bigoplus_{\ell \neq p} \hhh^i_{\eet}(X_j,\mu_{\ell^{\infty}}). \] 
It then follows from Lemma~\ref{r1lim-nul} that $\R^1\varprojlim_j \hhh^i_{\eet}(X_j, \overline{\so}^{\times})$ is zero and we obtain the desired isomorphism. 
\end{proof}


The terms $\varprojlim_j \big(\bigoplus_{\ell \neq p} \hhh^i_{\eet}(X_j, \mu_{\ell^{\infty}})\big)$ can be expressed as the completion of a direct sum for a well-chosen topology. Indeed, let us first endow $\hhh^i_{\eet}(U,\overline{\so}^{\times})$ with the discrete topology for any $U$ affinoid\footnote{As explained in the proof of Proposition~\ref{O-bar-stein0}, if the $\overline{\so}^{\times}$-cohomology $U$ is torsion in degree $i$, then $\hhh^i_{\eet}(U, \overline{\so}^{\times})$ is a quotient of $\varinjlim_n \hhh^i_{\eet}(U, \mu_n)$, where each $\mu_n$-cohomology group is finite so naturally discrete.}. For a general space $X$, consider the weakest topology such that for any \'etale affinoid $U \to X$ the restriction $\hhh^i_{\eet}(X,\overline{\so}^{\times}) \to \hhh^i_{\eet}(U,\overline{\so}^{\times})$ is continuous. 

\begin{lemma}\label{topology1}
Suppose that $X$ has an increasing covering by affinoids $X_j$'s. Then, the above topology is equivalent to the weakest topology making the restriction maps $\hhh^i_{\eet}(X,\overline{\so}^{\times}) \to \hhh^i_{\eet}(X_j,\overline{\so}^{\times})$ continuous for any $j$. 
\end{lemma}

\begin{proof}
Let $U \to X$ an \'etale morphism with $U$ affinoid. The $U \times_X X_j$ 's form an analytic covering of the quasi-compact $U$. For $j$ big enough, the natural map $U \times_X X_j \to U$ is an isomorphism and $U \to X$ naturally factors through $U \times_X X_j$ and $X_j$. We have a commutative diagram 
\[ \xymatrix{\hhh^i_{\eet}(X, \overline{\so}^{\times}) \ar[r] \ar[d] & \hhh^i_{\eet}(X_j, \overline{\so}^{\times}) \ar[ld] \\  \hhh^i_{\eet}(U, \overline{\so}^{\times})&} \]
where the oblique map is continuous.  
\end{proof}

\begin{corollary}\label{gr2}
Under the hypothesis of Proposition~\ref{O-bar-stein0}, $\hhh^i_{\eet}(X, \overline{\so}^{\times})$ identifies with the completion of 
\[ \bigoplus_{\ell \neq p} \hhh^i_{\eet}(X, \Q_{\ell}(1))/\hhh^i_{\eet}(X, \Z_{\ell}(1))\] 
for the topology defined above. 
\end{corollary}

\begin{proof}
By definition and Lemma~\ref{topology1}, the groups
\[N_j= {\rm ker}(\phi_j: \hhh^i(X, \overline{\so}^{\times}) \to \hhh^i(X_j, \overline{\so}^{\times}))  \]
defines a fundamental system of neighborhood of $0$ for the above topology. The completion is then 
\[   \varprojlim_j \big(\bigoplus_{\ell \neq p} \hhh^i_{\eet}(X,\mu_{\ell^{\infty}})\big) /N_j \simeq \varprojlim_j {\rm im}(\phi_j) \simeq \varprojlim_j \bigoplus_{\ell \neq p} \hhh^i_{\eet}(X_j, \Q_{\ell}(1))/\hhh^i_{\eet}(X_j, \Z_{\ell}(1)).  \] 
\end{proof}


\section{On torsion of $\overline{\so}^{\times}$-cohomology}\label{torsion}

Consider the class $\CC$ of smooth affinoid spaces with torsion $\overline{\so}^{\times}$-cohomology in degrees higher than $2$. Our expectation is that it is a sufficiently large class inside the category of affinoids. To develop this intuition we show some stability properties of $\CC$ and provide  some examples in this section. 

Let us write $(U_n)_n$ for the affinoids open in the usual Stein covering of the Drinfeld symmetric space (see Section~\ref{open-drinfeld} for a definition of the latter), the main theorem of this section reads as follows:
\begin{theorem}\label{thtorsion}
$\CC$ contains the opens $U_n$ and products of closed discs and closed annuli.
\end{theorem}

\subsection{Products of closed discs and closed annuli}

We begin by proving some kind of homotopy invariance of the class $\CC$.
More precisely, the goal of this subsection is to show the following lemma: 

\begin{lemma}\label{homot}
Let $\B(r)$ (respectively $\A(r,s)$) be a closed disc (respectively annulus) of dimension $1$ over $C$. If $X\in\CC$ then $X\times \B(r)\in \CC$ and $X\times \A(r,s)\in \CC$.
\end{lemma}

It is almost immediate to check that $\CC$ contains geometric points $\Spa(C,\so_C)$, and by a straightforward induction,  one gets

\begin{corollary}\label{torsion-disc1}
$\CC$ contains products of closed discs and closed annuli on $C$.
\end{corollary}

\begin{remark}\label{torsion-disc2}
Note that for those products of closed discs and closed annuli on $C$,  $\overline{\so}^{\times}$-cohomology is torsion even in degree higher than $1$ by vanishing of their Picard group and by Lemma~\ref{lem-heuer}\eqref{pdiv}.
\end{remark}

The proof of Lemma~\ref{homot} relies on these two intermediate results:  

\begin{lemma}
\label{induction}
Let $\B(r)$ (respectively $\A(r,s)$) be a closed disc (respectively annulus) of dimension one over $C$. Then, for any $X/C$ smooth affinoid and $i \ge 1$, we have 
\begin{align*}
 &\hhh^i_{\eet}(X \times \B(r), \overline{\so}^{\times}) \simeq  \hhh^i_{\eet}(X, \overline{\so}^{\times}), \\
 &\hhh^i_{\eet}(X \times \A(r,s), \overline{\so}^{\times}) \simeq \hhh^i_{\eet}(X, \overline{\so}^{\times}) \times \hhh^i_{\eet}(X, \Z[\frac{1}{p}]). 
 \end{align*}
\end{lemma}

\begin{lemma}
\label{torsion-Z}
Let $X/C$ be an affinoid, then for all $i \ge 1$, the cohomology groups $\hhh^{i}_{\eet}(X, \Q) $ are trivial whereas the  groups $\hhh^{i}_{\eet}(X, \Z[\frac{1}{p}]) $ are torsion. 
\end{lemma}

\begin{proof}[Proof of Lemma~\ref{induction}]
The proof for the closed disc is given in~\cite[Lemma~5.29]{EGN24}, we treat the case of the annulus in a similar way. To simplify the notations, in this proof we set $A:= \mathbf{A}(r,s)$. We denote by $\pi$ the natural projection $X \times A \to X$ and for $q>0$, we set ${\cal F}_q:=\R^q \pi_*\overline{\so}^{\times}$. It suffices to show that $\pi_*\overline{\so}^{\times} \simeq \overline{\so}^{\times} \times \Z[1/p]$ and that ${\cal F}_q =0$ for $q >0$. Indeed, using the Leray spectral sequence: 
\[ E_2^{p,q}=\hhh_{\eet}^p(X, \R^q\pi_*\overline{\so}^{\times}) \Rightarrow \hhh^{p+q}_{\eet}(X \times A, \overline{\so}^{\times}) \]  
this implies the result of the lemma.

The fact that  ${\cal F}_q$ is zero for $q >0$ can be checked on the stalks. Let $x := \Spa(L, L^+)$ be a geometric point of $X$. We claim that 
\[ ({\cal F}_q)_x \simeq \hhh^q_{\eet}(A \times \Spa(L, L^+) , \overline{\so}^{\times}). \]
The proof is the same as in~\cite[Lemma~5.29]{EGN24}: the argument can be reduced to checking that if $(Y_s)_s$ is a cofiltered system of smooth rigid spaces over $C$ and $Y\sim \varprojlim Y_s$ then $\hhh^*_{\eet}(Y, \overline{\so}^{\times}) \simeq \varinjlim_s \hhh^*_{\eet}(Y_s,\overline{\so}^{\times})$ and this was proved in~\cite[Proposition~3.2]{Heu25}.   

Moreover, we have that $\hhh^q_{\eet}(A \times \Spa(L, L^+) , \overline{\so}^{\times})$ is equal to  $\hhh^q_{\eet}(A \times \Spa(L, L^{\circ}) , \overline{\so}^{\times})$. Consider the exponential short exact sequence on the \'etale site of $A \times \Spa(L, L^{\circ})$: 
\[ 1 \to \so \xrightarrow{\rm exp} \so^{\times}[\frac{1}{p}] \to \overline{\so}^{\times} \to 1. \]
Since $\hhh^q_{\eet}(A \times \Spa(L, L^{\circ}) , \so)$ and $\hhh^q_{\eet}(A \times \Spa(L, L^{\circ}) , \so^{\times})[\frac{1}{p}]$ are trivial for $q \ge 1$ (the first one because $A \times \Spa(L, L^{\circ})$ is affinoid, the second one by \cite[Corollary 2.5.10]{dJ-vdP} for $q\ge 2$ and \cite[(3.25)]{vdp}), we obtain that $\hhh^q_{\eet}(A \times \Spa(L, L^{\circ}) , \overline{\so}^{\times})$ is zero, as wanted. 

Let us now prove that $\pi_*\overline{\so}^{\times} \simeq \overline{\so}^{\times} \times \Z[1/p]$. Choose a coordinate $T$ for the annulus $A$ and $Z\to X$ an étale cover. Then by \cite[Lemme 4.4.]{J1}, we have 
\begin{equation}
\label{o-bar-annulus}
{\so}^{\times}(Z \times A)/U(Z \times A) = {\so}^{\times}(Z )/U(Z) \times T^{\Z}.
\end{equation}
By Lemma~\ref{lem-heuer}\eqref{pdiv} ($Z$ and $Z\times A$ are quasi-compact), \[\hhh_{\eet}^0(Z, \pi_*\overline{\so}^{\times})=\hhh_{\eet}^0(Z \times \A(r,s), \overline{\so}^{\times}) =({\so}^{\times}(Z \times \A(r,s))/U(Z \times \A(r,s)))[1/p] = \overline{\so}^{\times}(Z) \times T^{\Z[1/p]},\]
as wanted. 
\end{proof}

\begin{proof}[Proof of Lemma~\ref{torsion-Z}]
Let $\pi : X \to {\rm Sp}(C)$ be the structural map. By \cite[Lemma~3.4.6]{dJ-vdP}, we know that ${\rm R}\pi_*\Q$ is overconvergent and by \cite[Corollary~3.3.2]{dJ-vdP}, we see that the stalks of ${\rm R}\pi_*\Q$ are torsion, hence they are zero (as they are also $\Q$-modules). This proves that the sheaf ${\rm R}\pi_*\Q$ is zero. Since we also have that $\hhh_{\eet}^i({\rm Sp}(C), \pi_*\Q)$ is zero in strictly positive degree (by the same argument), we obtain that $\hhh^{i}_{\eet}(X, \Q) $ vanishes for $i >0$. 

For $\Z[\frac{1}{p}]$-coefficients, consider the exact sequence: 
\[ 0 \to \Z[\frac{1}{p}] \to \Q \to \bigoplus_{\ell \neq p} (\Q_{\ell}/ \Z_{\ell}) \to 0, \]
where the sum is taken over prime numbers $\ell$ different from $p$.
Since $\hhh^i_{\eet}(X, \Q)$ vanishes for $i \ge 1$ by the previous discussion, we obtain an isomorphism (using that $X$ is quasi-compact),
\[ \hhh^{i}_{\eet}(X, \Z[\frac{1}{p}]) \simeq \hhh^{i-1}_{\eet}(X,\bigoplus_{\ell \neq p} \Q_{\ell}/ \Z_{\ell}) \simeq \bigoplus_{\ell \neq p} \limi_n \hhh^{i-1}_{\eet}(X, \Z/ {\ell}^n\Z). \] 
Since each term $\hhh^{i-1}_{\eet}(X, \Z/ {\ell}^n\Z)$ is torsion, this concludes the proof.
\end{proof}

\subsection{Opens of the Drinfeld symmetric space}\label{open-drinfeld}
 Now we explain how to propagate this result to affinoids for which the geometry is deeply related to products of closed discs and closed annuli, like for example the opens $(U_n)_n$ of the Drinfeld symmetric space. In particular, this will finish the proof of Theorem~\ref{thtorsion}. Let us first recall how these affinoids $U_n$ are defined and introduce some notations.

Let us denote by $\mathbb{H}^d_K$ the Drinfeld symmetric space of dimension $d$ over $K$, i.e. the rigid analytic space defined by 
\[\mathbb{H}_K^d=\mathbb{P}_K^d \setminus \bigcup_{H \in {\cal H}} H, \] 
where ${\cal H}$ is the set of $K$-rational hyperplanes in ${\bb P}_K^d$ and ${\bb P}_K^d$ is the rigid analytic projective space of dimension $d$ over $K$. 
We also write $\mathbb{H}^d:=\mathbb{H}_K^d \times_K C$ for its base change to $C$.  

Recall that ${\cal H}$ can be identified with the set 
\[ \{{\rm Ker}(l_a), \; a \in K^{d+1} \setminus \{ 0 \}\} \simeq {\bb P}^d(K) \]
where $l_a$ is the linear form $b \in C^{d+1} \mapsto \sum_{0 \le i \le d} a_i b_i$. If $a$ is unimodular (i.e., ${\rm max}_{0 \le i \le d}( | a_i |)=1$) then we define $l_a^{(n)}$ as the map from $(\so_C/\varpi^n)^{d+1}$ to $\so_C/\varpi^n$ induced by $l_a$. We set 
\[{\cal H}_n:=\{{\rm Ker}(l^{(n)}_a), \; a \in K^{d+1} \setminus \{ 0 \} \text{ unimodular}\} \simeq {\bb P}^d(\so_K/ \varpi^n) \] 
such that ${\cal H} = \varprojlim_n {\cal H}_n$. For $\varepsilon >0$ and $H={\rm Ker}(l_{a_H})$ in ${\cal H}$ we denote by $H(\varepsilon)$ the  open tube of radius $\varepsilon$ around $H$, that is 
\[
H(\varepsilon) = \{ z \in {\bb P}^d(C), \; |l_{a_H}(z)|< \varepsilon |z_i| \quad \forall i\}. \] 
We write $H(\varepsilon)^c$ for the complement in the rigid space ${\bb P}_{K}^d$. The goal of this section is to prove:   

\begin{proposition}
\label{torsion-cov-dr}
For $n \in \N$, let $U_{n,K}:=\bigcap_{H \in {\cal H}_n} H(| \varpi |^n)^c$ and $U_n:=U_{n,K} \times_K C$. Then $\hhh^i_{\proeet}(U_n, \overline{\so}^{\times})$ is torsion for all $i \ge 1$.
\end{proposition}

The reasoning to prove the above proposition is similar to the one from~\cite[Section~5.2]{J1}. Let us first prove the following general result: 
 
\begin{lemma}
\label{torsion-cov-dr-lem1}
 Let $X$ be a rigid analytic space and let ${\cal U}=\{ U_i, i\in I\}$ be a family of open sets in $X$. Assume that $\hhh^{\bullet}(-)$ is a cohomology theory that satisfies Mayer-Vietoris exact sequence. Then,  
 \begin{enumerate}
     \item If for all $k \ge 1$ and all finite subset $J \subset I$ the groups $\hhh^k( \bigcap_{j \in J} U_j)$ are torsion, then for all $k \ge |J|$, the cohomology groups of finite unions $\hhh^k(\bigcup_{j \in J} U_j)$ are torsion.
     \item If for all $k \ge |J|$ and all finite subset $J \subset I$ the groups $\hhh^k( \bigcup_{j \in J} U_j)$ are torsion, then the cohomology groups of all the finite intersections $\hhh^k(\bigcap_{j \in J} U_j)$ are torsion.
 \end{enumerate}    
\end{lemma}

\begin{proof}

We can assume $I=\{ 1, \dots, n\}$. Both points will be proved by induction on $n$. Let us first prove (1). If $n=1$, this is trivial. Take $n \ge 2$ and let us assume that the result of (1) is true for $I=\{ 1, \dots , n-1\}$. It suffices to prove that the cohomology of $\bigcup_{i=1}^n U_i$ is torsion in all degree $\ge 1$. By hypothesis, we have an exact sequence: 
\[ \hhh^{k-1}(U^{(n-1)} \cap U_n) \to \hhh^k(\bigcup_{i=1}^n U_i) \to \hhh^k(U^{(n-1)}) \oplus \hhh^k(U_n), \] 
where $U^{(n-1)}= \bigcup_{i=1}^{n-1} U_i$. Since we have $U^{(n-1)} \cap U_n= \bigcup_{i=1}^{n-1} V_i$ with $V_i:= U_i \cap U_n$, its cohomology groups are torsion (by induction hypothesis). The same is true for $\hhh^k(U^{(n-1)})$ and $\hhh^k(U_n)$. This proves that $\hhh^k(\bigcup_{i=1}^n U_i)$ is torsion, as wanted. 

To prove the second point, we can assume $I=\{ 1, \dots, n\}$ and we are reduced to prove that the cohomology of $\bigcap_{i=1}^n U_i$ is torsion in all degree $\ge 1$. Setting $V_i:= U_i \cap U_n$, we see that it is enough to show that the cohomology of $\bigcap_{i=1}^{n-1} V_i$ is torsion. By induction, it suffices to show that $\hhh^k(\bigcup_{j \in J} V_j)$ is torsion for any finite subset $J \subset \{ 1, \dots , n-1\}$ and $k\ge |J|$. Let us write $U^J:=\bigcup_{j \in J} U_j$ and $V^J:=\bigcup_{j \in J} V_j=U^J\cap U_n$. We have an exact sequence:  
\[\hhh^k(U^{J}) \oplus \hhh^k(U_n) \to \hhh^k(V^J) \to \hhh^{k+1}(U^J \cup U_n). \] But by hypothesis, we have that the groups $\hhh^k(U^{J})$, $\hhh^k(U_n)$ and $\hhh^{k+1}(U^J \cup U_n)$ are all torsion. We deduce that $\hhh^k(V^J)$ is torsion.
\end{proof}

Proposition~\ref{torsion-cov-dr} follows then from the second point of Lemma~\ref{torsion-cov-dr-lem1} combined with the following result:  

\begin{lemma} \label{torsion-cov-dr-lem2}
Let ${\cal A}_n$ be a subset of ${\cal H}_n$. Then for $k \ge 1$, the groups $\hhh_{\eet}^k( \bigcup_{H \in {\cal A}_n} H(| \varpi |^n)^c, \overline{\so}^{\times})$ are torsion.  
\end{lemma}

\begin{proof}
This will also follow from Lemma~\ref{torsion-cov-dr-lem1}. We fix ${\cal A}_n$ a subset of  ${\cal H}_n$ and write $X({\cal A}_n):= \bigcup_{H \in {\cal A}_n} H(| \varpi |^n)^c$. From~\cite[Section~1]{scst} (see also \cite[Section~2.2]{J1}), we know that there exist $t \le d$, a fibration $f: X({\cal A}_n) \to {\bb P}^t$ and an admissible covering ${\cal V}=\{ V_i \}_{i=0, \dots, t}$ of ${\bb P}^t$ such that $f^{-1}(V_i) \simeq V_i \times {\B}^{d-t}(| \varpi|^{\beta_i})$ for some integer $\beta_i$ and such that $V_i$ is a  product of closed discs and annuli. This gives an admissible covering $f^*({\cal V})=\{ f^{-1}(V_i)\}_i$ of $X({\cal A}_n)$ such that the intersections of the $f^{-1}(V_i)$ on $f^*({\cal V})$ are product of closed discs and annuli. Hence, from Corollary~\ref{torsion-disc1} and Remark~\ref{torsion-disc2} we know that the groups $\hhh^k_{\eet}( \bigcap_{j \in J} f^{-1}(V_j), \overline{\so}^{\times})$'s for $J \subset \{ 0, \dots, t \}$ are torsion for all $k \ge 1$. Using the first point of Lemma~\ref{torsion-cov-dr-lem1}, we thus obtain that $\hhh^k_{\eet}(X({\cal A}_n), \overline{\so}^{\times})$ is torsion for all degree $k \ge 1$.           
\end{proof}

\section{Flatness results for $\Z/\ell^n\Z$-cohomology}\label{flat-0}
The goal of this section is to prove that the affinoids appearing in the standard Stein covering of the Drinfeld space satisfy the hypothesis (${\bf H}_2'$) from Theorem~\ref{main-theorem}. 
\subsection{Preliminaries} \label{flat-prelim}
In this section, we recall a few facts about the Bruhat-Tits building of the group ${\rm PGL}_{d+1}(K)$ and its relation with Drinfeld's space.  


\subsubsection{The simplicial complex $\BC\TC$}

 Let us denote $\BC\TC= \bigcup_{k \ge 0} \BC\TC_k$ the Bruhat-Tits building associated to the group ${\rm PGL}_{d+1}(K)$, where $\BC\TC_k$ is the set of simplexes of dimension $k$ in $\BC\TC$.
A $(k+1)$-tuple of vertices 
$\sigma=\{s_0,\cdots, s_k\}\subset \BC\TC $ is a $k$-simplex in $\BC\TC_k$ if and only if, up to permuting the $s_i$'s, we can find, for any $i$, lattices $M_i$ of $K^{d+1}$ with $s_i=[M_i]$ such that \[M_{-1}=\varpi M_k\subsetneq M_0\subsetneq M_1\subsetneq\cdots\subsetneq M_k.\]

If $s$ is a vertex and $\sigma$ a simplex, we define
\[{\rm St}(s)=\bigcup_{\tau \ni s} \mathring{\tau} \quad \text{ and } \quad {\rm St}(\sigma) = \bigcap_{t \in \sigma} {\rm St}(t), \]
\[\bar{\rm St}(s)=\bigcup_{\tau \ni s} {\tau} \quad \text{ and } \quad \bar{\rm St}(\sigma) = \bigcap_{t \in \sigma} \bar{\rm St}(t). \]
Note that we have the relation: 
\begin{equation}
    \label{sigma-rond}
    \mathring{\sigma} = {\rm St}(\sigma) \setminus \bigcup_{\tau \supset \sigma} {\rm St}(\tau).  
\end{equation}
We can endow $\mathcal{BT}_0$ with a combinatorial distance $\delta$ given by the length of minimal paths in the graph $\mathcal{BT}_1$. Fix $s_0=[\so_K^{d+1}]$ and define the "closed and open balls" for $\delta$ for $n\ge 1$: 
\[ B(n)= \bigcup_{\sigma : \delta(s_0, s) \le n-1} \bar{\rm St}(s) \quad \text{ and } \quad \mathring{B}(n)=\bigcup_{s : \delta(s_0, s) \le n-1} {\rm St}(s),  \] 
set also $B(0)=\mathring{B}(0)=s_0$.

\subsubsection{Orbit action of $G$}\label{orbitG}
Denote by $G$ the group $\mathrm{GL}_{d+1}(K)$. It acts transitively on $\BC\TC_0$ and the standard vertex $s_0=[\so_K^{d+1}]$ has stabilizer group $K^*{\rm GL}_{d+1}(\so_K)=: K^*\Gamma^{(0)}_{s_0} $.  This group $\Gamma^{(0)}_{s_0}$ admits a filtration $(\Gamma^{(n)}_{s_0})_{n\ge 1}$ by congruence subgroups $\Gamma^{(0)}_{s_0}:=\{ g \in {\rm GL}_{d+1}(\so_K)\; | \; g = {\rm Id} \pmod{\varpi^n} \}$. For any other vertex $s=gs_0$ we can also define $\Gamma^{(n)}_{s}=g\Gamma^{(n)}_{s_0}g^{-1}$ which is independent of the element $g$. For a general simplex $\sigma=\{s_0,\cdots, s_k\}$, we define $\Gamma^{(n)}_{\sigma}$ as the subgroup generated by $\Gamma^{(n)}_{s_1} \cup \Gamma^{(n)}_{s_2} \cup \dots \cup \Gamma^{(n)}_{s_r}$. Usually we will work with a family of congruence subgroups $\Gamma^{(n)}_{\sigma}$  for the same integer $n$ that we will drop in the previous notation.

The topological realization of $\BC\TC$ will be denoted by $|\BC\TC|$. The $k$-simplexes $\sigma$ can be seen as compact spaces $|\sigma|$ in the topological realization such that we have $|\BC\TC|=\bigcup_{\sigma\in \BC\TC} |\sigma |$. 
The interior of $| \sigma|$ will be written $\mathring{\sigma}=\sigma \backslash \bigcup_{\sigma'\subsetneq\sigma}\sigma'$.


For a simplex $\sigma=(M_0, \dots, M_k)$ as above, setting 
 $\bar{M}_i=M_i/ M_{i-1},$
we obtain a flag $0\subsetneq\bar{M}_0\subsetneq \bar{M}_1\subsetneq\cdots\subsetneq \bar{M}_k\simeq {\bf F}^{d+1}$. We define 
 $d_i={\rm dim}_{\bf F} (\bar{M}_i)-1$ et $e_i=d_{i+1}-d_{i}.$
 We will then say the simplex $\sigma$ is of type $(e_0, e_1,\cdots, e_k)$.
 
 Let us consider a basis $(\bar{f}_0,\cdots, \bar{f}_d)$ adapted  to the flag, i.e. such that $\bar{M}_i=\left\langle \bar{f}_0,\cdots , \bar{f}_{d_i}\right\rangle$ for any $i$. For any lift $(f_0,\cdots ,f_d)$ of $(\bar{f}_0,\cdots ,\bar{f}_d)$ in $M_0$, we have
\begin{equation} \label{def-N} M_i=\left\langle f_0,\cdots, f_{d_i }, \varpi f_{d_i+1},\cdots, \varpi f_d \right\rangle=N_0\oplus \cdots \oplus N_i\oplus \varpi (N_{i+1} \oplus \cdots\oplus N_k), \end{equation}
  where 
  $$N_i=\left\langle f_{d_{i-1}+1},\cdots, f_{d_{i}} \right\rangle$$ with $d_{-1}=-1$. 
  
\subsubsection{The reduction map $\tau$}
The Bruhat-Tits building $\mathcal{BT}$ allows us to define different open sets of ${\bb H}^d_K$ (and of ${\bb H}^d$ its base change to $C$). More precisely, consider the map 
$\tau: {\bb H}^d(C)\to |\BC\TC|$ (see \cite[§I.4]{boca} and \cite[§2.1]{wa} for more details).  
Let $\sigma\in\BC\TC_{k}$ be a simplex of type $(e_0, e_1,\cdots, e_k)$ and define:
\begin{align*} 
&{\bb H}_{\sigma}^d:=\tau^{-1} (\sigma), \;\; {\bb H}_{\mathring{\sigma}}^d:=\tau^{-1} (\mathring{\sigma}),  \\
&U_{n}:= \tau^{-1}(B(n)), \;\; \mathring{U}_{n}:=\tau^{-1}(\mathring{B}(n)),\\
&U_{\sigma}:= \tau^{-1}({\rm St}(\sigma)).
\end{align*}
Here, $U_n$ is affinoid, $\mathring{U}_n$ is Stein, if $s$ is a vertex then $U_s=\mathring{U}_{1}$ is the tube of an irreducible component of the special fiber and ${\bb H}_{\mathring{\sigma}}^d$ is the smooth locus of $U_{\sigma}= \bigcap_{s \in \sigma} U_s$ (by~\eqref{sigma-rond}). Moreover, we have $U_n \subset \mathring{U}_{n+1} \subset U_{n+1}$ and the $U_n$'s define a Stein covering for ${\bb H}^d$. Note that the $U_n$'s defined above coincide with the ones define in Section~\ref{open-drinfeld}.

\begin{lemma}\label{u-n}\cite[Proposition~8.2]{J1}
Let $\ell$ be a prime number different from $p$. The maps
\begin{align*}
    & \hhh^i_{\dr}(\mathring{U}_{n+1}) \to \hhh^i_{\dr}({U}_{n}^{\dagger}) \\
    & \hhh^i_{\eet}(\mathring{U}_{n+1}, \Z/{\ell}^k\Z) \to \hhh^i_{\eet}({U}_{n},\Z/{\ell}^k\Z)
\end{align*}
are isomorphisms for any $i\ge 0$. 
\end{lemma}

Consider the open
\[C_{r}:= \{x=(x_1,\cdots, x_r)\in \B^r|\, \forall a\in \OC^{r+1}_K\backslash \varpi \OC^{r+1}_K, \  1=|\left\langle (1, x),a \right\rangle| \},\]
\[A_{k}=\{y=(y_0,\cdots, y_{k-1})\in \B^k|\,  1>|y_{k-1}|>\cdots>|y_0|>|\varpi|\}.\]
 For $\sigma$ of type $(e_0, \dots, e_k)$, we have morphisms 
$$ {\bb H}_{\mathring{\sigma}}^d \rightarrow C_{e_i -1}, \,\, 
  [z_0,\cdots, z_d]  \mapsto  (\frac{ z_{d_{i}+1} }{ z_{d_{i}} },\frac{ z_{d_{i}+2} }{ z_{d_{i}} },\cdots, \frac{ z_{d_{i+1}-1} }{ z_{d_{i}} }) \text{ and} $$
 $${\bb H}_{\mathring{\sigma}}^d \rightarrow A_{k},\,\, 
  [z_0,\cdots, z_d]  \mapsto  (\frac{ z_{d_{0}} }{ z_{d} },\frac{ z_{d_{1}} }{ z_{d} },\cdots, \frac{ z_{d_{k-1}} }{ z_{d} }).$$
It is proved in \cite[6.4]{ds}  that the above 
    morphisms induce isomorphisms
\begin{align}    \label{decompoH}
 {\bb H}_{\mathring{\sigma}}^d \simeq A_{k} \times \prod_{i=0}^k{C_{e_i-1}}=: A_{k} \times C_{\sigma}.
 \end{align}
 
 \subsubsection{Euclidean properties of $\BC\TC$}
  
  Given a basis $(\bar{f}_0,\cdots, \bar{f}_d)$, the set of simplexes with vertices of the form $(\varpi^{\alpha_0} f_0 \so_K, \dots,\varpi^{\alpha_d} f_d \so_K )$ defines an apartment in $\mathcal{BT}$, i.e. an euclidean space $\R^{d+1}/(\R \cdot (\alpha_0, \dots, \alpha_d))$ such that the vertices corresponds to $\Z^{d+1}/(\Z \cdot (\alpha_0, \dots, \alpha_d))$. This euclidean structure can be glued to the whole building $\mathcal{BT}$. We denote by $d$ the induced distance. The geodesics between two points $x$ and $y$ are given by the lines in an apartment ${\cal A}$ containing $x$ and $y$ (such an apartment exists).
  
  \begin{definition}
  Let $Y$ be a subset in $|\mathcal{BT}|$. We say that $Y$ is geodesically contractible\footnote{This terminology is unfortunate and should be called instead geodesically convex but we decided to follow the pre-existing literature on the topic (see \cite[p184]{ron} or \cite[§2 Corollary 3]{scst2}).} if any geodesics $[x,y]$ is contained in $Y$ for $x,y$ in $Y$. 
  \end{definition}
  
  \begin{proposition}
\label{B-convex}  $B(n)$ is geodesically contractible. 
  \end{proposition}

\begin{proof}
Since any geodesics is contained in an apartment, it suffices to show that $B(n) \cap {\cal A}$ is convex in ${\cal A}$ for any apartment ${\cal A}$. 

Let us fix ${\cal A}$ an apartment and $(f_0, \dots, f_d)$ an associated basis. Set $M:= \so_K^{d+1}$. We would like to characterize the vertices $M_1=(\varpi^{\alpha_0} f_0, \dots, \varpi^{\alpha_d} f_d)$ in $B(n) \cap {\cal A}$.

We write $(m_{i,j})_{i,j}$ for the base change matrix such that $f_i = \sum_{j=0}^d m_{i,j} e_j$ and $(\widetilde{m}_{i,j})_{i,j}$ its inverse. For $a$ an integer, we have 
\[ \varpi^a e_i = \sum_{j=0}^d \varpi^{a- \alpha_j} \widetilde{m}_{i,j} (\varpi^{\alpha_j} f_j), \]
so the set of integers $a$ such that $\varpi^aM$ is contained in $M_1$ is equal to the set of integers $a$ such that 
\[ a \ge {\rm max}_{0 \le i,j \le d} \{ \alpha_j - v(\widetilde{m}_{i,j})\}. \]
Similarly, since we have 
\[\varpi^{\alpha_i} f_i = \sum_{j=0}^d \varpi^{\alpha_i} m_{i,j} e_j, \] 
the set of $b$ such that $M_1$ is contained in $\varpi^bM$ is equal to the set of integers $b$ such that 
\[ b \le {\rm min}_{0 \le i,j \le d} \{ v(m_{i,j}) + {\alpha_i} \}. \] 
By Lemma~\ref{distance} below, we have 
\[ \delta(M,M_1)= {\rm max}_{0 \le i,j \le d} \{ \alpha_j - v(\widetilde{m}_{i,j})\}-{\rm min}_{0 \le i,j \le d} \{ v(m_{i,j})+ {\alpha_i} \}.  \] 

For $(x_i) \in \R^{d+1}/\R \cdot (1, \cdots , 1)$, we define 
\begin{align*} f((x_i)_i)&:= {\rm max}_{0 \le i,j \le d} \{ x_j - v(\widetilde{m}_{i,j})\}-{\rm min}_{0 \le i,j \le d} \{ v(m_{i,j})+ {x_i} \} \\
&=:{\rm max}_{0 \le i,j \le d} \; g_{i,j}(x) - {\rm min}_{0 \le i,j \le d} \; h_{i,j}(x),\end{align*}
with $g_{i,j}$ and $h_{i,j}$ affine. Consider the set 
\[ C:= \{ (x_i) \in \R^{d+1}/\R \cdot (1, \cdots , 1) \; | \; f((x_i)_i) \le n \}. \]
Since $g_{i,j}$ and $-h_{i,j}$ are convex for all $i,j$, then ${\rm max}_{0 \le i,j \le d} \; (g_{i,j})$ and ${\rm max}_{0 \le i,j \le d} \; (-h_{i,j})$ are convex and so is also $f$ (as a sum of convex functions). This proves that $C$ is convex. In particular, $C$ contains $B(n) \cap {\cal A}$.  

We have the following fact: 

\begin{lemma}
For all maximal simplex $\sigma \in {\cal A}$, the function $f|_{\sigma}$ is affine.
\end{lemma}

\begin{proof}
Let $\sigma=(s_0, \dots, s_d)$ be a maximal simplex in ${\cal A}$. Up to permutation of the canonical generating family of $\R^{d+1}/ \R \cdot (1, \dots, 1)$, we can suppose that $s_0$, $s_1$, \dots, $s_d$ have the following description as elements of $\Z^{d+1}/ \Z \cdot (1, \dots, 1)$:
\begin{align*}
    &s_0=(\alpha_0, \dots, \alpha_d) \\
    &s_i=(\alpha_0+1, \dots, \alpha_{i-1}+1, \alpha_i, \dots, \alpha_d).
\end{align*}
Let $x= \sum_{i=0}^d t_i s_i= (\alpha_i + \sum_{k=i}^d t_k)_i$ be an element in $\sigma$ with $\sum_{i=0}^d t_i=1$ and $t_i\ge 0$ for all $i$. Then 
\begin{align*}  g_{i,j}(\alpha) \le & g_{i,j}(x) < g_{i,j}(\alpha)+1 \\ 
 -h_{i,j}(\alpha)+1 < & -h_{i,j}(x) \le -h_{i,j}(\alpha) 
\end{align*}
The $g_{i,j}(\alpha)$'s are in $\Z$, so ${\rm max} \; g_{i,j}(x)$ is reached for $(i_0,j_0)$ such that $ g_{i_0,j_0}(\alpha)$ is maximal. Among those $(i_0, j_0)$, take $(i_{\rm min}, j_{\rm min})$ such that $i_{\rm min}$ is maximal. This yields the maximality of $\sum_{k=i_{\rm min}}^d t_k$, thus the maximality of $g_{i_{\rm min}, j_{\rm min}}(x)$ on the pairs $(i,j)$, and it only depends on $\alpha$ and not on $x$. Similarly, ${\rm max} \; -h_{i,j}(x)$ is reached for a unique $(i_{\rm max}, j_{\rm max})$ which does not depend on $x$, i.e. $f|_{\sigma}= g_{i_{\rm min}, j_{\rm min}}|_{\sigma}-h_{i_{\rm max}, j_{\rm max}}|_{\sigma}$, which is clearly affine. This concludes the proof of the lemma.     
\end{proof}

Let us now consider $\sigma$ a maximal simplex in ${\cal A}$. On vertices of $\sigma$, $f$ admits at most two values $k$, $k+1$ (let us assume $k$ is always reached). If $k <n$, then $\sigma$ is contained in $B(n) \cap {\cal A}$ and so in $C$. When $k >n$, $\sigma$ does not meet $B(n) \cap {\cal A}$ and since $f$ is bigger than $n+1$ by affineness of $f$, $\sigma$ cannot meet $C$ either. Now suppose $k=n$, consider $\tau = \sigma \cap B(n) \cap {\cal A}$. Then $f|_{\tau}$ is equal to the constant $n$. On the complement $\sigma \setminus \tau$, the values of $f$ are strictly bigger than $n$ (again by affineness). In conclusion, $C \cap \sigma= B(n) \cap {\cal A} \cap \sigma$. 

We obtain $B(n) \cap {\cal A}=C$, which is convex. 


\end{proof} 

\begin{lemma}
\label{distance}
Let $M_0$ and $M_1$ be two vertices in $\mathcal{BT}$. Then 
\[\delta(M_0, M_1)= {\rm min} \{ a-b \; | \; \varpi^a M_0 \subset M_1 \subset \varpi^b M_0 \}. \] 
\end{lemma}

\begin{proof}
Suppose $\varpi^a M_0 \subset M_1 \subset \varpi^b M_0$ and consider the sequence $(\varpi^s M_0+M_1)_{b \le s \le a}$. These modules define a path in the graph $\mathcal{BT}_1$ of length $a-b \ge \delta(M_0, M_1)$. 

Consider $(M_0, \dots , M_r)$ a path in $\mathcal{BT}_1$ such that for all $0\le i \le r-1$, we have $\varpi M_{i+1} \subset M_i \subset M_{i+1} $. By an immediate induction, for all $0 \le s \le r$, 
\[ \varpi^sM_s \subset M_0 \subset M_s. \]
Setting $s=r$ yields the other inequality. 
\end{proof}

\begin{corollary}\label{convcomp}
Any geodesically contractible simplicial subcomplex of $\mathcal{BT}$ is a union of compact geodesically contractible simplicial subcomplexes. 
\end{corollary}

\begin{proof}
For any $Y \subset \mathcal{BT}$, $(Y \cap B(n))_n$ gives the desired family.
\end{proof}

\subsection{Coefficient systems}

In this subsection, we would like to understand the $\Z/{\ell^n}\Z$-cohomology (for $\ell \neq p$) of the affinoids $U_{\sigma}$ defined above and study the \v{C}ech cohomology theory associated to this covering. We endow $\bigcup_{k\ge 1} \mathcal{BT}_k$ with the order given by the inclusion and we see it as a category.

\begin{definition}
Let ${\cal C}$ be a category. 
\begin{enumerate}
\item An homological (respectively cohomological) ${\cal C}$-coefficient system on $\mathcal{BT}$ is a covariant (respectively contravariant) functor from $\bigcup_{k\ge 1} \mathcal{BT}_k$ to the category of $\Lambda$-modules, i.e. it is the data, for any simplex $\sigma$, of a module $M_{\sigma}$ together with compatible maps $M_{\tau} \to M_{\sigma}$ (respectively $M_{\sigma} \to M_{\tau}$) if $\tau \subset \sigma$. 

\item A ${\cal C}$-coefficient system $(M_{\sigma})_{\sigma}$ is $G$-equivariant if for any $g$ in $G$, there exists a natural transformation $T_g$ from $(M_{\sigma})_{\sigma}$ to $(M_{g \cdot \sigma})_{\sigma}$ such that 
\[T_{gh}(\tau)=T_h(g \cdot \tau) \circ T_g(\tau) \] 
\[T_{\rm Id}(\tau)={\rm Id}_{M_\tau} \] 
for any $g,h$ in $G$ and any $\tau$ in $ \mathcal{BT}$. 

\item Assume that ${\cal C}$ is an abelian category. If $M:=(M_{\sigma})_{\sigma}$ is an homological (respectively cohomological) ${\cal C}$-coefficient system and $Y \subset \mathcal{BT}$, we define a complex $C_Y^{\bullet}(M)$ with terms given by 
\[ C_Y^{k}(M)= \bigoplus_{\sigma \in Y_k} M_{\sigma} \quad \text{ (respectively } C_Y^{k}(M)= \prod_{\sigma \in Y_k} M_{\sigma})  \]
where $Y_k:=Y \cap \mathcal{BT}_k$, and the differential maps are given by the alternating sums of the maps $M_{\tau} \to M_{\sigma}$ (respectively $M_{\sigma} \to M_{\tau}$) for $\tau$ a face of codimension $1$ in $\sigma$. Moreover, we say that $M$ is $Y$-acyclic if the complex $C_Y^{\bullet}(M)$ is acyclic. 

\item Let $\Lambda$ be a ring. If ${\cal C}$ is one of the categories $\Lambda-{\rm mod}$ or $\Lambda-{\rm grad}$ and if $M$ is an homological (respectively cohomological) ${\cal C}$-coefficient system then the functor ${\rm Hom}_{\Lambda}(M, \Lambda):= \sigma \mapsto {\rm Hom}_{\Lambda}(M_{\sigma}, \Lambda)$ is a cohmological (respectively homological) ${\cal C}$-coefficient system and this operation preserves $Y$-acyclic objects. 
\end{enumerate}
\end{definition}


Let $H$ be a locally profinite group, we say that a $\Lambda[H]$-module $M$ is smooth if the $H$-stabilizer of any element $m\in M$ is open.

\begin{example}
\begin{enumerate}
\item The map $\hhh^*_{\rm dR}: \sigma \mapsto \hhh^*_{\rm dR}(U_{\sigma})$ (respectively $\hhh^*_{\eet, n}: \sigma \mapsto \hhh^*_{\eet}(U_{\sigma}, \Z/\ell^n\Z)$) endowed with the natural restriction maps defines a cohomological $C$- (respectively $\Z/\ell^n\Z$-) graded coefficient system.


\item 
Let fix $n \ge 1$ an integer. If $V$ is a smooth $\Lambda[G]$-module, we can associate to $V$ the homological coefficient system $V^{(n),\bullet}: \sigma \mapsto V^{\Gamma^{(n)}_{\sigma}}$, where $\Gamma^{(n)}_{\sigma}$ is the congruence subgroup of $G$ introduced in Section~\ref{orbitG}.  %
\end{enumerate}
\end{example}



We will in fact construct a third type of examples of graded $\Lambda$-coefficient systems $A^*_{\Lambda}(\bullet)\simeq A^*_{\Z}(\bullet)\otimes \Lambda$ for any ring $\Lambda$ and these objects will allow us to relate the two examples mentioned above. We will give their construction in the next subsection and establish the following three results which are the main steps of the argument presented in this section.

\begin{lemma}
\label{comp}
For any $p$-torsionfree ring $\Lambda$ and any simplex $\sigma\in\BC\TC$, the $\Lambda$-module $A_{\Lambda}^k(\sigma)$ is finite free and there are isomorphisms of graded coefficient systems:
\[ A_C^*(\bullet) \simeq \hhh^*_{\dr} \quad \text{and} \quad A_{\Z/\ell^n\Z}^*(\bullet)\simeq \hhh^*_{\eet,n}.\] 
\end{lemma}

\begin{lemma}
\label{flat-4}
For $\Lambda$ a $\Z[1/p]$-algebra, there exists a smooth $\Lambda[G]$-module $V_k$ such that the coefficient system $A_{\Lambda}^k(\bullet)\simeq{\rm Hom}_{\Lambda}(V^{(1),\bullet}_k, \Lambda)$.  
\end{lemma}

\begin{lemma}
\label{flat-5}
Let $n\ge 1$, $A$ be a $\Z[1/p]$-algebra and a principal ideal domain and $\Lambda \in \{{\rm Frac}(A), A/(a), A \}$ with  $a\in A$. Consider $V$ a smooth $A[G]$-module  generated by $V^{\Gamma^{(n)}_{s_0}}$ and denote by $V_{\Lambda}$ its base change to $\Lambda$. If $Y$ is geodesically contractible, then  $V_{\Lambda}^{(n),\bullet}$ is $Y$-acyclic.  
\end{lemma}

The proofs of Lemmas \ref{comp} and \ref{flat-4} will occupy  next section (see Proposition~\ref{flat-comp},  Theorem~\ref{flat-7} and Theorem~\ref{dual-A}) whereas Lemma \ref{flat-5} will be established in the Section \ref{acyclic1}. 
\begin{corollary}\label{h0}
For $\Lambda$ a $\Z[1/p]$-algebra and $Y$  geodesically contractible,  $A^*_{\Lambda}(\bullet)$ is $Y$-acyclic. Moreover, when $Y$ is compact, the degree zero cohomology is finite free and satisfies \begin{equation}\label{h0_2}
\hhh^0(C^{\bullet}_Y(A^*_{\Lambda}(\bullet)))\simeq \hhh^0(C^{\bullet}_Y(A^*_{\Z[1/p]}(\bullet)))\otimes \Lambda.
\end{equation}
\end{corollary}

\begin{proof}
By Lemmas \ref{flat-4} and \ref{flat-5}, we know that $A^*_{\Z[1/p]}(\bullet)$ is $Y$-acyclic and its degree zero cohomology is finite free over $\Z[1/p]$ when $Y$ is compact, as a submodule of a finite free $\Z[1/p]$-module (Lemma \ref{comp}).   Moreover we have in this case \[C^{\bullet}_Y(A^*_{\Lambda}(\bullet))=C^{\bullet}_Y(A^*_{\Z[1/p]}(\bullet))\otimes \Lambda\] for any ring $\Lambda$ (we have used that each $A_{\Z[1/p]}^k(\sigma)$ is finite free over $\Z[1/p]$) and we deduce for all $i\ge 1$, by  universal coefficient theorem (\cite[Theorem 5.5.3]{span}), \[\hhh^i(C^{\bullet}_Y(A^*_{\Lambda}(\bullet)))\simeq \hhh^i(C^{\bullet}_Y(A^*_{\Z[1/p]}(\bullet)))\otimes \Lambda,\] which is also finite free. 
\end{proof}

\begin{corollary}\label{flat-2}
 For $Y$  geodesically contractible, we have for all $i \ge 0$,  
 \begin{align*}
 & \hhh^i_{\rm dR}(\bigcup_{s\in Y_0} U_s) \simeq {\rm Ker}\big( \prod_{s \in Y_0} \hhh^i_{\rm dR}(U_{s}) \to \prod_{a \in Y_1} \hhh^i_{\rm dR}(U_{a})\big)  \\
 \text{(resp.} \quad  & \hhh^i_{\eet}(\bigcup_{s\in Y_0} U_s, \Z/\ell^n \Z) \simeq {\rm Ker}\big( \prod_{s \in Y_0} \hhh^i_{\eet}(U_{s}, \Z/\ell^n \Z) \to \prod_{a \in Y_1} \hhh^i_{\eet}(U_{a}, \Z/\ell^n \Z)\big) \text{ ),}
 \end{align*}
 and those modules are finite free over their respective base  ring when $Y$ is compact.
 
Moreover, the space $\mathbb{H}^d$ together with the covering given by the affinoids $U_n$ defined in Proposition~\ref{torsion-cov-dr} satisfies hypothesis $(\mathbf{H}_2)$ from Theorem~\ref{main-theorem}.
\end{corollary}

\begin{proof}
By Lemma \ref{comp}, $C_Y^{\bullet}(\hhh^*_{\rm dR})\simeq C^{\bullet}_Y(A^*_{C}(\bullet))$ and $C_Y^{\bullet}(\hhh^*_{\eet, n})\simeq C^{\bullet}_Y(A^*_{\Z/\ell^n \Z}(\bullet))$.  As the differentials of both complexes are compatible with the gradation, we obtain a double complex whose associated spectral sequence is the \v{C}ech spectral sequence corresponding to the covering of $\bigcup_{s\in Y_0} U_s$ given by the $U_s$'s. Lemma~\ref{h0} implies that  this spectral sequence degenerates and the first statement of the corollary follows. 

Applying this to $Y=B(n)$, we obtain the result concerning the space $\mathbb{H}^d$ by Lemma \ref{u-n}, Equation \ref{h0}\eqref{h0_2} and the functoriality in $Y$ of $C^{\bullet}_Y(A^*_{\Lambda}(\bullet))$.
\end{proof}

When $\Lambda=C$ and $Y= \mathcal{BT}$, Lemmas \ref{flat-4} and \ref{flat-5} have been proved by de Shalit\footnote{He actually calculates the de Rahm cohomology of ${\bb H}^d_K$ over $K$ and relates it to the coefficient system $A^*_{K}(\bullet)$. Base changing everything to $C$ has no incidence on his arguments.} in~\cite{ds}. Our strategy follows closely his arguments. In order to ensure that his results stay true for more general coefficients, we re-write below the proofs in details.  

\subsection{Construction of coefficient systems and broken circuits}\label{brok}
In this subsection, we construct auxiliary coefficient systems $A_{\Lambda}^k(\bullet)$, $\widetilde{A}_{\Lambda}^k(\bullet)$,  coming from the theory of hyperplane arrangements and whose description is purely combinatorial. We will see in Proposition~\ref{flat-comp} that the systems $\hhh^*_{\eet,n}$ and $\hhh^*_{\rm dR}$ are actually isomorphic to one of those for a good choice of $\Lambda$.  The general case could have been established by copying the arguments from Sections 2 and 3 of \cite{ds}, but we have rather chosen to first deduce the case $\Lambda= \Z$ from the case $\Lambda=C$ proved in {\em loc. cit.}. The result in the general case then follows relatively easily.

We denote by $\widetilde{E}$ the free exterior algebra $\bigwedge^* \Lambda[{\cal A}]$, where ${\cal A}$ is a subset of ${\cal H}$ (most of the time ${\cal A}= {\cal H}$). Then $\widetilde{E}=\bigoplus_r \widetilde{E}^r$ is a $\Lambda$-graded algebra endowed with the wedge product $\wedge$. If $s$ is a sequence of elements of $\mathcal{A}$, we write $e_s$ for the associated vector in $\widetilde{E}$. For $s$, $s'$ two sequences, we denote by $s \cdot s'$ their concatenation so that we have $e_{s} \wedge e_{s'}= e_{s \cdot s'}$ and $s^{(i)}$ is the sequence with $r-1$ terms obtained from $s:=(s_1, \cdots, s_k)$ by removing the $i$-th term (with the convention $e_{\emptyset}=1$).

Note that when ${\cal A}= {\cal H}$, then the group $G$ acts on $\widetilde{E}$ via its action on ${\cal H}$.

We recall the following facts: 
\begin{itemize}
    \item The algebra $\widetilde{E}$ is free with generators $e_{s}:= e_{s_1} \wedge \cdots \wedge e_{s_r}$ for any $s:=(s_1, \cdots, s_r)$ in ${\cal A}^r$. 
    \item It admits a differential $d:\widetilde{E} \to \widetilde{E} $ of degree $-1$ defined by $d(e_{s})= \sum_{i=1}^k (-1)^i e_{s^{(i)}}$.
    \item The complex $\widetilde{E} \to \widetilde{E} \to \widetilde{E}$ concentrated in degrees $0,1,2$ has trivial cohomology in degree $1$ (\cite[(2.6)]{ds}). 
\end{itemize}

Let us write $E:= \widetilde{E}^{d=0}= d(\widetilde{E})$. Note that $E$ is generated as a graded ring by the elements $a-b= d(a \wedge b)$ of degree $1$. Moreover, for any  $a$ in ${\cal A}$, the differential $d:\widetilde{E} \to {E}$ admits a section $x \mapsto a \wedge x$ (\cite[(2.5)]{ds}).  

\begin{definition}\label{i-sigma}
Let $\sigma:=(M_0, \cdots, M_k)$ be a simplex. We define $I(\sigma)$ to be the graded ideal in $\widetilde{E}$ generated by the elements $d(e_{s})$ where $s=(s_j)_{1 \le j \le k}$ is in $\bigcup_{k \ge 0} ({\cal A} \cap (M_i \setminus M_{i+1}))^r$ for some $i$ and such that its components $(s_j)_{1 \le j \le k}$ are linearly dependent in $M_i/M_{i+1}$. We set 
$\widetilde{A}_{\Lambda}(\sigma) := \widetilde{E}/I(\sigma)$ and write $A_{\Lambda}(\sigma)$ for its projection onto $E$, i.e. 
\[A_{\Lambda}(\sigma) = (E + I(\sigma))/I(\sigma) = E/(E \cap I(\sigma)).\]
\end{definition}

\begin{remark}\label{a-fini}
Note that for a fixed simplex $\sigma:=(M_0, \cdots, M_k)$, two subsets ${\cal A}$ and ${\cal A}'$ of ${\cal H}$ give rise to the same algebras $\widetilde{A}_{\Lambda}(\sigma)$ and ${A}_{\Lambda}(\sigma)$ when the projections of ${\cal A} \cap M_i$ and ${\cal A}' \cap M_i$ onto $M_i/M_{i+1}$ coincide for all $i$. In particular, when we work over a simplex $\sigma$, we can assume  ${\cal A}$ to be finite and this shows that  $\widetilde{A}_{\Lambda}(\sigma)$ and ${A}_{\Lambda}(\sigma)$ are  finitely generated over $\Lambda$.
\end{remark}

We endow the algebras $A_{\Lambda}(\sigma)$ and $\widetilde{A}_{\Lambda}(\sigma)$ with the graded structure coming from the ones of $E$ and $\widetilde{E}$. We write $A^k_{\Lambda}(\sigma)$ and $\widetilde{A}^k_{\Lambda}(\sigma)$ for the corresponding $k$-th graded pieces.  
A direct computation gives: 
\begin{lemma}\cite[Proposition~2.2]{ds}
\label{flat-6}
If $\sigma:=(M_0, \cdots, M_k)$ is a simplex, then $I(\sigma)=I(M_0) + \cdots+ I(M_k)$. 
\end{lemma}
It follows from the lemma that $(A_{\Lambda}(\sigma))_{\sigma}$ and $(\widetilde{A}_{\Lambda}(\sigma))_{\sigma}$ both define $\Lambda$-graded coefficient systems with graded pieces  $(A^k_{\Lambda}(\sigma))_{\sigma}$ and $(\widetilde{A}^k_{\Lambda}(\sigma))_{\sigma}$.
Moreover, if ${\cal A}={\cal H}$, it is easy to check that $(A_{\Lambda}(\sigma))_{\sigma}$ is $G$-equivariant. 

\begin{proposition}\label{flat-comp}
There are isomorphisms of graded coefficient systems:
\[ A_C(\bullet) \simeq \hhh^*_{\dr} \quad \text{and} \quad A_{\Z/\ell^n\Z}(\bullet)\simeq \hhh^*_{\eet,n} .\]
\end{proposition}

\begin{proof}
For $\Lambda=C$ (resp. $\Lambda=\Z/\ell^n\Z$) set $\hhh^*:=\hhh^*_{\dr}$ (resp. $\hhh^*:=\hhh^*_{\eet,n}$).  
Let us first assume that $\sigma=s$ is a vertex. The input of the arguments of Orlik-Solomon \cite{orlsol} has been axiomatised in this case by Schneider and Stuhler in \cite[\textsection 2]{scst} and they are satisfied by $\hhh^*$ in both cases. In particular, their spectral argument directly gives an isomorphism 
\begin{equation}\label{Asommet} A_{\Lambda}(s) \xrightarrow{\sim} \hhh^*(U_s). \end{equation}
 
Let now $\sigma:=(M_0, \dots, M_k)$ be any simplex. 
Consider the decomposition~\eqref{decompoH}, it induces a restriction map \[\hhh^*(U_{\sigma}) \to \hhh^*(C_{e_0-1} \times \cdots \times C_{e_k-1} \times A_k) \] which is an isomorphism by \cite[Theorem 2.4]{GK2} in the de Rham case and by \cite[Lemma 5.6]{zhe} in the \'etale case (${\bb H}^d$ has a semi-stable integral model and $U_\sigma$ and ${\bb H}^d_{\mathring{\sigma}}$ are related to components of the special  fiber and their smooth locus  by \eqref{sigma-rond}). 

The key idea is to observe that each component $C_{e_i-1}$ is itself the tube over a vertex of a certain Drinfeld symmetric space but of lower dimension, namely $e_i-1$, where we could apply the previous discussion. Moreover, $A_k$ is an annulus for which we can explicitly compute the cohomology. Let us develop this strategy.  

If $X$ is one of the spaces $C_{e_i-1}$, $A_k$, then $X$ is an open of ${\bb P}^{{\rm dim}(X)}$. Define $\widetilde{E}_X=\bigwedge^* \Lambda[{\cal A}_X]$ where ${\cal A}_X$ is ${\bb P}^{{\rm dim}(X)}(K)$ when $X$ is one of the $C_{e_i-1}$ and ${\cal A}_{A_k}=\{ {\rm ker}(y_i), 0\le i \le k \}$, where the $y_i$'s are the coordinates of $A_k$. It admits a differential $d$ defined similarly to the one of $E$ and we set $E_X=(\widetilde{E}_X)^{d=0}$. Note that $\widetilde{E} \simeq \bigotimes_X \widetilde{E}_X$ as graded $\Lambda$-algebras. We have a natural map $E_X \to \hhh^*(X) $ with image generated by the logarithmic classes. Let us write $A_X:=E_X/I_X$ this quotient where $I_{C_{e_i-1}}$ is the ideal from Definition \ref{i-sigma} for the corresponding vertex and $I_{A_k}=0$.

Using K\"unneth formula and isomorphism \eqref{Asommet} in lower dimension, we have graded isomorphisms: 
\begin{align*}
    \hhh^*(U_{\sigma}) & \simeq \hhh^*(C_{e_0-1}) \otimes \cdots \otimes \hhh^*(C_{e_k-1}) \otimes \hhh^*(A_k) \\
    & \simeq A_{C_{e_0-1}} \otimes \cdots \otimes A_{C_{e_k-1}} \otimes A_{A_k}.
    \end{align*}
This proves that the natural map $E \to \hhh^*(U_{\sigma})$ is surjective with kernel equal to the graded ideal generated by the $I_X$'s. It remains to prove that this ideal is equal to $I(\sigma)$. 

Let us consider the modules $N_0, \dots, N_k$ from \eqref{def-N}. By construction, we have an identification $\iota_i : N_i/\varpi N_i \xrightarrow{\sim} M_i/M_{i-1}$ and maps $\Z[{\bb P}(M_i/M_{i-1})]^0 \to \so^*(U_{\sigma})$ and $\Z[{\bb P}(N_i/\varpi N_{i})]^0 \to \so^*(C_{e_i-1})$ which fit into a commutative diagram: 
\[\xymatrix{\Z[{\bb P}(M_i/M_{i-1})]^0 \ar[r] & \so^*(U_{\sigma}) \ar[d]\ar[r] & \hhh^1(U_{\sigma}) \ar[d]\\
\Z[{\bb P}(N_i/\varpi N_{i})]^0 \ar[r] \ar[u]_{\rotatebox{90}{$\sim$}}^{\iota_i} & \so^*(C_{e_i-1}) \ar[r] & \hhh^1(C_{e_i-1}) } \]
By Definition~\ref{i-sigma}, $I(\sigma)$ admits a system of generators in bijection with 
\begin{align*}
    \bigcup_{0 \le i \le k} \{ \text{ linearly dependent families of } {\bb P}(M_i/M_{i-1}) \} \\
    \xrightarrow[\iota_i]{\sim} \bigcup_{0 \le i \le k} \{ \text{ linearly dependent families of } {\bb P}(N_i/\varpi N_{i}) \}.  
    \end{align*} 
Similarly, each term in the latter union is in bijection with a system of generators of the ideal $I_{C_{e_i-1}}$. We obtain  
 that $I(\sigma)$ is generated by the $I_X$'s, as wanted. 

\end{proof}
 
By construction, $\widetilde{A}_{\Lambda}(\sigma)$ are generated by the $e_s$'s.
The broken circuit theorem (Theorem~\ref{flat-7}) states that we can extract a basis from it by considering elements $e_s$ such that $s$ satisfies some maximality condition for a certain order on ${\cal A}$. To define it properly, we will need to assume ${\cal A}$ to be finite, which is possible by Remark \ref{a-fini}. 

Let $\sigma= (M_0, \cdots, M_k)$ be a simplex and let us fix a total order $>$ on ${\cal A}_i:= {\cal A} \cap (M_i \setminus M_{i+1})$. Consider a chain $s=(a_0 > a_1 > \cdots > a_n)$ in ${\cal A}_i^n$ for a certain $i$. For any $1 \le j \le n$, we can write $s$ as $a_0 \cdot a_1 \cdot \cdots \cdot a_{j-1} \cdot s_j$ with $s_j$ a sequence in ${\cal A}^{n+1-j}$. We say that $s$ is ($\sigma$-)special if for all $ 0 \le j \le n$, we have $a_j={\rm max}( {\cal A} \cap (M_{i+1} + \langle s_j \rangle) )$ where $\langle s_j \rangle$ denotes the span over $\so_K$ of the terms of the sequence $s_j$. More generally, if $s=s_0 \cdot s_1 \cdots s_{r-1} \cdot s_r$ with $s_i$ in $\bigcup_{k \ge 0} {\cal A}^k_i$ for all $i$ then we say that $s$ is special if $s_i$ is special for any $0 \le i \le r$. We will also say, by abuse of terminology, that $e_s$ is special whenever $s$ is special. 

\begin{remark}
If we consider another presentation $(M_i, \cdots , M_r, M_0, \cdots, M_{i-1})$ of the simplex $\sigma$, then for any special element $s=s_0 \cdot s_1 \cdots s_{r-1} \cdot s_r$, we have $e_{s_0 \cdot s_1 \cdots s_{r-1} \cdot s_r} =(-1)^{i}e_{s_i \cdot s_{i+1} \cdots s_{r} \cdot s_0 \cdots \cdot s_{i-1}}$ and we see that up to a sign, the two definitions of being special coincide.   
\end{remark}

\begin{theorem}
\label{flat-7}
The special elements form a $\Lambda$-basis of $\widetilde{A}_{\Lambda}(\sigma)$. Moreover, $A_{\Lambda}(\sigma)$ admits as a basis the family of elements $e_{a \cdot s}$ with $s$ special and $a= {\rm max}({\cal A} \cap (M_0 \setminus M_1) )$.  
\end{theorem}

\begin{proof}
Using the section $E \to \widetilde{E}, x \mapsto e_a \wedge x$ of the differential, one can easily see that the result for $A_{\Lambda}(\sigma)$ follows from the one for $\widetilde{A}_{\Lambda}(\sigma)$. In \cite{ds}, the claim is proved when $\Lambda=C$ and it suffices to prove the case $\Lambda = \Z$ since $\widetilde{A}_{\Lambda}(\sigma) =\widetilde{A}_{\Z}(\sigma) \otimes \Lambda$. The fact that the family $\{ e_s \}_{s \text{ special}}$ is free over $\Z$ follows from the inclusion $\widetilde{A}_{\Z}(\sigma) \subset \widetilde{A}_{C}(\sigma)$. We need to prove that it is a generating family. To do that, we follow the arguments in the proof of \cite[Theorem~2.5]{ds}. 

We want to prove that any element $e_s$ with $s$ in $\bigcup_{k \ge 0} {\cal A}^k_i$ can be written as a linear combination of special elements $e_t$ with $M_{i+1}+\langle t \rangle \subset M_{i+1}+\langle s \rangle  $. We proceed by induction on the length of the sequence $s$. 

Let $s=a_0 \cdot s'$ a sequence of length $n+1$. By induction, $e_{s'}$ can be written as a linear combination of special elements. Choose $e_{s''}$ one of these elements. We obtain a decomposition of $e_s$ containing the term $e_{a_0 \cdot s''}$. Consider $a_0':= {\rm max}({\cal A}_i \cap (M_i+\langle  a_0\cdot s''\rangle))$. Then $d(e_{a_0'\cdot a_0\cdot s''})$ is in $\widetilde{A}_{\Z}(\sigma)$ (the elements are linearly dependent) and we can express $e_{a_0\cdot s''}$ as a combination (over $\Z$) of terms of the form $e_{a_0'\cdot a_0\cdot s''^{(j)}}$ and $e_{a_0's''}$. We use again the induction to write $e_{a_0\cdot s''^{(j)}}$ and $e_{s''}$ as a linear combination of special elements and we choose $e_{s'''}$ one of these special elements. So we have a decomposition of $e_s$ with one of the terms being $e_{a_0'\cdot s'''}$. We need to prove that $e_{a_0'\cdot s'''}$ is special and that $M_{i+1}+\langle a_0'\cdot s'''\rangle$ is  contained in $M_{i+1}+\langle s\rangle$. 

By construction $M_{i+1}+\langle a_0'\cdot s'''\rangle$ is either contained in $M_{i+1}+\langle a_0'\cdot a_0\cdot s''^{(j)}\rangle$ or in $M_{i+1}+\langle a_0'\cdot s''\rangle $, and both of those are contained inside $M_{i+1}+\langle a_0\cdot s''\rangle $ which in turn is contained inside $M_{i+1}+\langle s\rangle $. Since $a_0$ is maximal in $M_{i+1}+\langle a_0\cdot s''\rangle $, it is maximal in all the subsets containing it and in particular it is maximal in $M_{i+1}+\langle a_0'\cdot s'''\rangle $. It follows that $e_{a_0'\cdot s'''}$ is special.           
\end{proof}

\begin{corollary} \label{flat-free}
For all $i \ge 0$ and all $\sigma$ in $\mathcal{BT}$, the modules $\hhh^i_{\dr}(U_{\sigma})$ and $\hhh^i_{\eet}(U_{\sigma}, \Z/{\ell}^n\Z)$ are finite free.  
\end{corollary}

We will now exhibit linear forms on algebras $A^k_{\Lambda}(\tau)$. For $\sigma$ a simplex, we consider the function $i_{\sigma} : {\cal A} \to [0,k]$ defined by the following property: for $a$ in ${\cal A}$, $i_{\sigma}(a)$ is the unique integer $i$ such that $a$ meets $M_i \setminus M_{i+1}$. We extend the signature operator to all maps of subsets of the integers $f: [0, n] \to [0, m]$ by setting ${\rm sgn}(f)=0$ if $f$ is not bijective. We define 
\[l_{\sigma}: s=(a_0, \cdots , a_m) \mapsto {\rm sgn}(i_{\sigma}(a_0), \cdots, i_{\sigma}(a_m)). \] 
The usual properties satisfied by the signature operator ensure that if $p$ is a permutation of $[0,m]$, then $l_{\sigma}(p\cdot s)= l_{\sigma}(s){\rm sgn}(p)$. We extend the maps $l_{\sigma}$'s to $\widetilde{E}$ and $E$ by setting $l_{\sigma}(e_s):=l_{\sigma}(s)$. Moreover, for any simplex $\tau$ contiguous to $\sigma$, we claim that the map $l_{\sigma}$ vanishes on $E^{k+1} + I^{k+1}(\tau)$, in particular, it factors through $A^k_{\Lambda}(\tau)$. In the case where $\Lambda=C$, this claim was proved in \cite[Lemma 2.4]{ds}. But the generators of $E^{k+1} + I^{k+1}(\tau)$ are $\Z$-linear combinations of elements $e_s$ when $\Lambda=\Z$ whose image by $l_{\sigma}$ in $\{-1,0,1\}$ does not depend on the ring $\Lambda$. This proves the case $\Lambda= \Z$. The general case follows by scalar extension. 

These linear forms are closely related to the basis of $A^k_{\Lambda}(\tau)$ given by Theorem~\ref{flat-7}. More precisely, given a $\tau$-special chain $t$, we can construct, following the proof of \cite[Theorem~2.5]{ds} (see formula (2.29) in loc. cit.), a simplex $\sigma_t$ contiguous to $\tau$ such that
\begin{equation}
    \label{flat-eq1}
l_{\sigma_t}(e_t)=1 \quad \text{ and } \quad l_{\sigma_t}(e_s)=0 \text{ if } s>t. 
\end{equation}
These relations only depend on the combinatorial properties of $\tau$ and $\sigma_t$ and thus, they are satisfied for any ring $\Lambda$ as long as they are true for a non-torsion ring (here, the field $C$). Applying Theorem~\ref{flat-7}, we get for any ring $\Lambda$: 

\begin{lemma}\cite[Corollary~2.7]{ds}
\label{flat-8}
The family $l_{\sigma_t}$  defines a basis of $A^k_{\Lambda}(\tau)^*$, where $t$ runs over the $\tau$-special elements of $S$. 
\end{lemma}

\begin{remark}
In fact in the original proof of de Shalit in \cite{ds}, the relations~\eqref{flat-eq1} are used in the proof of Theorem~\ref{flat-7} to obtain that the elements $e_{a \cdot s}$ are independent. Note that this does not create any circular argument in our proof: as we have accepted as true and are using all the results from~\cite{ds}, our argument is not contingent on any interdependence between the results of de Shalit. 
\end{remark}

From now on we assume ${\cal A}= {\cal H}$. Let us define ${\cal L}_r(\Lambda)$ as the subspace of ${\rm Hom}(E^r, \Lambda)$ linearly generated by the forms $l_{\sigma}$ with $\sigma$ of dimension $r$ (note that they satisfy relations coming from harmonicity conditions, see \cite[\textsection3.1]{ds}). Then ${\cal L}_r(\Lambda)$ is a smooth $\Lambda[G]$-module: this is because by definition of the action, ${\rm Stab}(l_{\sigma})$ contains ${\rm Stab}(\sigma)$ which is open.  

We will now prove the following theorem which is the crucial point of this section: 

\begin{theorem}\label{dual-A}
Let $\tau$ be a simplex of dimension $r$ and assume that $\Lambda$ does not have $p$-torsion. Then, we have a decomposition
\[ {\cal L}_k(\Lambda)^{{\rm Stab}(\tau)} = \bigoplus \Lambda \cdot l_{\sigma} \]
where the sum is taken over the $l_{\sigma}$'s such that the family $\{ l_{\sigma} \}$ forms a dual basis of $A^k_{\Lambda}(\tau)$. In particular,  $A^k_{\Lambda}(\bullet)\simeq \Hom({\cal L}_k(\Lambda)^{(1),\bullet}, \Lambda)$.    
\end{theorem}

Before proving the theorem, let us recall the following standard result: 
\begin{lemma} \label{exinv}
Let $H$ be a pro-$p$-group and $\Lambda$ a $\Z[\frac{1}{p}]$-algebra. Then $M \mapsto M^{H}$ defines an exact functor from the category of smooth $\Lambda[H]$-modules to the category of $\Lambda$-modules. 
\end{lemma}
\begin{corollary}\label{comtensinv}
Let $A \to B$ be a morphism of $\Z[\frac{1}{p}]$-algebras and $H$ be a pro-$p$-group. For any smooth $A[H]$-module $M$, the natural map 
\[M^H \otimes_A B \to (M \otimes_A B)^H \]
is an isomorphism. 
\end{corollary}

\begin{proof}
First note that the result is true for $M=A[H/N]$ where $N \subset H$ is an open normal subgroup. 

For a general $M$, we can find a resolution 
\[ M_2 \to M_1 \to M \to 0 \]
with $M_1, M_2$ direct sum of modules of the form $A[H/N]$. The lemma follows then from the following commutative diagram with exact rows by Lemma~\ref{exinv}: 
\[\xymatrix{M_2^H \otimes_A B \ar[r] \ar[d]_{\sim} & M_1^H \otimes_A B \ar[r] \ar[d]_{\sim} & M^H \otimes_A B \ar[r] \ar[d] & 0\\
(M_2 \otimes_A B)^H \ar[r]  & (M_1 \otimes_A B)^H \ar[r]  & (M \otimes_A B)^H \ar[r] & 0 
} \] 
\end{proof}

\begin{proof}[Proof of Theorem \ref{dual-A}]
From \cite[Theorem~5.4]{ds}, we know that the theorem is true for $\Lambda=C$. Since taking the ${\rm Stab}$-invariant is left exact, we have 
\[ {\cal L}_{r}(\Z[\frac{1}{p}])^{{\rm Stab}(\tau)} = {\cal L}_{r}(C)^{{\rm Stab}(\tau)} \cap {\cal L}_{k}(\Z[\frac{1}{p}]) =\bigoplus \Z[\frac{1}{p}] \cdot l_{\sigma}.\] 
Let us now assume that $\Lambda$ is a $\Z[\frac{1}{p}]$-algebra. It follows from the Corollary \ref{comtensinv} that 
\[{\cal L}_r(\Lambda)^{{\rm Stab}(\tau)} \simeq {\cal L}_r(\Z[\frac{1}{p}])^{{\rm Stab}(\tau)} \otimes_{\Z[\frac{1}{p}]} \Lambda \simeq \bigoplus \Lambda \cdot l_{\sigma},\]
as wanted.


Assume now that $\Lambda$ does not have $p$-torsion, it injects into  $\Lambda[\frac{1}{p}]$ for which we have by the previous discussion        \[ {\cal L}_{r}(\Lambda[\frac{1}{p}])^{{\rm Stab}(\tau)} = \bigoplus \Lambda[\frac{1}{p}] \cdot l_{\sigma} \] and we can conclude \[{\cal L}_{r}(\Lambda)^{{\rm Stab}(\tau)} = {\cal L}_{r}(\Lambda[\frac{1}{p}])^{{\rm Stab}(\tau)} \cap {\cal L}_{r} (\Lambda) =\bigoplus \Lambda \cdot l_{\sigma}.\]
\end{proof}

\subsection{Acyclicity of coefficient systems coming from smooth representations}\label{acyclic1}
In this section we fix $A$ a principal ideal domain, $\Lambda \in \{A, {\rm Frac}(A), A/(a), a \in A \}$, $Y$ a geodesically contractible subset of $\mathcal{BT}$ and $V$ a smooth $A[G]$-module generated by its $\Gamma^{(n)}_{s_0}$-invariant elements for some $n\ge 1$. Our next goal is to prove Lemma~\ref{flat-5} for such  $\Lambda$, $Y$ and $V$. Recall that our goal is to prove Lemma~\ref{flat-5}, i.e. we want to show that $V_{\Lambda}^{(n), \bullet}$ is $Y$-acyclic. The case where $\Lambda$ is an algebraically closed field with characteristic $0$ has been established in~\cite{scst2}, we follow here their strategy. 

Let us first consider the  smooth $A[G]$-module ${\cal C}_c(G/\Gamma^{(n)}_{s_0})$ of functions $G/\Gamma^{(n)}_{s_0} \to A$ with compact support. We prove the following lemma:

\begin{lemma}\label{acyclic-1}
If $V:= {\cal C}_c(G/\Gamma^{(n)}_{s_0})$ then $V_{\Lambda}^{(n),\bullet}$ is $Y$-acyclic. 
\end{lemma}

\begin{proof}
We set $T:=G/\Gamma^{(n)}_{s_0}$ and define $T_{\sigma}:= \Gamma^{(n)}_{\sigma} \backslash T$. We consider the ${\cal C}$-coefficient system $\sigma \mapsto {\cal C}_c(T_{(\sigma)}^{\bullet})$, where ${\cal C}$ is the category of simplicial $\Lambda$-modules and  
$T_{(\sigma)}^{\bullet}$ is the simplicial set
\[ \xymatrix{ \cdots  \ar@<4.5pt>[r] \ar@<1.5pt>[r] \ar@<-1.5pt>[r] \ar@<-4.5pt>[r]& T \times_{T_{\sigma}} T \times_{T_{\sigma}} T \ar@<3pt>[r] \ar[r] \ar@<-3pt>[r] & T \times_{T_{\sigma}} T \ar@<1.5pt>[r] \ar@<-1.5pt>[r] & T  } \]
We write $K(T^{\bullet}_{(\sigma)})$ for the associated Kan complex. When $\sigma$ varies, one gets  a coefficient system of complexes $K(T^{\bullet}_{(\bullet)}):=(K(T^{\bullet}_{(\sigma)}))_\sigma$ and we are then interested in the complex $C_Y^{\bullet}(K(T^{\bullet}_{(\bullet)}))$:  
\[\small{ \xymatrix{ 0 \ar[r] & \bigoplus_{\sigma \in Y_d} {\cal C}_c(T) \ar[r] \ar[d] & \cdots \ar[r] & \bigoplus_{\sigma \in Y_1} {\cal C}_c(T) \ar[r] \ar[d]  & \bigoplus_{\sigma \in Y_0} {\cal C}_c(T) \ar[d] \\ 
0 \ar[r] & \bigoplus_{\sigma \in Y_d} {\cal C}_c(T \times_{T_{\sigma}} T) \ar[r] \ar[d] & \cdots \ar[r] & \bigoplus_{\sigma \in Y_1} {\cal C}_c(T \times_{T_{\sigma}} T) \ar[r] \ar[d]  & \bigoplus_{\sigma \in Y_0} {\cal C}_c(T \times_{T_{\sigma}} T) \ar[d] \\
0 \ar[r] & \bigoplus_{\sigma \in Y_d} {\cal C}_c(T \times_{T_{\sigma}} T\times_{T_{\sigma}} T) \ar[r] \ar[d] & \cdots \ar[r] & \bigoplus_{\sigma \in Y_1} {\cal C}_c(T \times_{T_{\sigma}} T\times_{T_{\sigma}} T) \ar[r] \ar[d]  & \bigoplus_{\sigma \in Y_0} {\cal C}_c(T \times_{T_{\sigma}} T\times_{T_{\sigma}} T). \ar[d] \\
& \vdots & & \vdots &\vdots } }\] 
By construction, the columns are acyclic and we want to prove that the first row is acyclic as well. To do that, it suffices to establish the acyclicity of  all the other rows. 

We see the sets $T^m_{(\sigma)}$ from the $m$-th row as subsets of $T^m$. For any element $t$ in $T^m$, we define $\mathcal{BT}^{(t)}$ as the set of all simplexes $\sigma$ in $\mathcal{BT}$ such that $t$ is not in $T^m_{(\sigma)}$. We will use this set below to construct an intermediate complex.

Let us consider the morphism of simplicial sets $T^m_{(\sigma)} \to T^m$, where the term on the right is the constant simplicial set. We have the following exact sequence of complexes:
\[0 \to K(T^m_{\bullet}) \to K(T^m) \to K(T^m \setminus T^m_{\bullet}) \to 0.  \]
Since $T^m$ is constant, the associated complex is acyclic. Hence, to prove that $K(T^m_{\bullet})$ is acyclic, it suffices to check $K(T^m \setminus T^m_{\bullet})$ is so. As $\bigcup_{\sigma \in Y_k} T^m \setminus T^m_{\sigma} = \bigcup_{t \in T^m} \mathcal{BT}^{(t)}_k \cap Y $, the $k$-th term of this complex is given by 
\begin{equation}
    \label{sum-btt}
 \bigoplus_{\sigma \in Y_k} \bigoplus_{t \in T^m \setminus T^m_{\sigma}} \Lambda = \bigoplus_{t \in T^m} \bigoplus_{\sigma \in \mathcal{BT}_k^{(t)}\cap Y} \Lambda.
\end{equation}
One can easily check that 
the differential sends $\bigoplus_{\sigma \in \mathcal{BT}_k^{(t)}\cap Y} \Lambda$ to $\bigoplus_{\sigma \in \mathcal{BT}_{k-1}^{(t)}\cap Y} \Lambda$, so the formulas~\eqref{sum-btt} give a decomposition of the complex as a direct sum indexed by the elements of $T^m$. Since $Y$ and $\mathcal{BT}^{(t)}$ ar geodesically contractible (by hypothesis and \cite[\textsection 1, Lemma]{scst2}), so is $\mathcal{BT}^{(t)} \cap Y$. We thus obtain that each direct factor of $K(T^m \setminus T^m_{\bullet})$ is acyclic, and the acyclicity of $K(T^m \setminus T^m_{\bullet})$  itself follows. 
\end{proof}

We are now able to deduce Lemma \ref{flat-5} from this special case. Take $A$, $\Lambda$, $Y$, $V$ and $n\ge 1$ as before. We would like to prove the acyclicity of the complex $C^{\bullet}_Y(V_{ \Lambda}^{(n),\bullet})$ for $V_{\Lambda}:= V \otimes_A \Lambda$ where $\Lambda \in \{{\rm Frac}(A), A, A/(a), a \in A\}$. Note that by Corollary~\ref{comtensinv} one has the identity $C_{Y_i}^{\bullet}(V_{\Lambda}^{(n),\bullet})\simeq C_{Y_i}^{\bullet}(V^{(n),\bullet})\otimes\Lambda$. 

Write $Y =\varinjlim_i Y_i$ with $Y_i$ compact and geodesically contractible (see Corollary \ref{convcomp}). Then $C_Y^{\bullet}(V_{\Lambda}^{(n),\bullet}) = \varinjlim_i C_{Y_i}^{\bullet}(V_{\Lambda}^{(n),\bullet})$ and it suffices to show that the $C_{Y_i}^{\bullet}(V_{\Lambda}^{(n),\bullet})$'s are acyclic.

Suppose then  that $Y$ is  compact. In that case, the terms of the complex $C_Y^{\bullet}(V_{\Lambda}^{(n),\bullet})$ are free $\Lambda$-modules  of finite type. 

Let us first deal with the cases $\Lambda={\rm Frac}(A)$ and $\Lambda =A/\mathfrak{p}$ with $\mathfrak{p}$ a prime ideal of $A$. In those cases, $\Lambda$ is a field, let us write $\overline{\Lambda}$ for its algebraic closure. Then we can find a resolution of $V_{\overline{\Lambda}}$ by direct sums of ${\cal C}_c(G/\Gamma^{(n)}_{s_0})$ and the acyclicity of ${\cal C}_c(G/\Gamma^{(n)}_{s_0})$ follows from Lemma \ref{acyclic-1}. Since $\Lambda \to \overline{\Lambda}$ is faithfully flat, we deduce the acyclicity of $C_Y^{\bullet}(V_{\Lambda}^{(n),\bullet})$ by scalar extension.

Suppose now that $\Lambda$ is $A/\mathfrak{p}^k$. We proceed by induction on $k$. The case $k=1$ was established above. Let $k \ge 2$ and assume that $C^{\bullet}_Y(V_{A/\mathfrak{p}^{k-1}}^{(n),\bullet})$ is acyclic. By considering the long exact sequence associated to: 
\[ 0 \to C^{\bullet}_Y(V_{A/\mathfrak{p}^{k-1}}^{(n),\bullet}) \to C^{\bullet}_Y(V_{A/\mathfrak{p}^{k}}^{(n),\bullet}) \to C^{\bullet}_Y(V_{A/\mathfrak{p}}^{(n),\bullet}) \to 0, \] 
we obtain that the cohomology groups of $C^{\bullet}_Y(V_{A/\mathfrak{p}^{k}}^{(n),\bullet})$ are trivial.

If $a$ is an element of $A$ which is not a unit, we can decompose $(a)$ as a product of powers of prime ideals $\mathfrak{p}_1^{k_1} \cdots \mathfrak{p}_d^{k_d}$ and we know that each one of the complexes $C^{\bullet}_Y(V_{A/\mathfrak{p}_i^{k_i}}^{(n),\bullet})$ is acyclic. By Chinese remainders theorem, we obtain that $C^{\bullet}_Y(V_{A/(a)}^{(n),\bullet})$ is acyclic.  

Now assume that $\Lambda$ is $A$. By acyclicity of the complex $C^{\bullet}_Y(V_{{\rm Frac}(A)}^{(n),\bullet})$, we know that the cohomology groups of $C^{\bullet}_Y(V^{(n),\bullet})$  are  $A$-torsion. Moreover, they are all finitely generated because we have chosen $Y$ to be compact. Thus, we can find an element $a_i\in A$ for which the group $\hhh^i(C^{\bullet}_Y(V^{(n),\bullet}))$ is an $A/a_iA$-module. In particular, any cocycle of $C^i_Y(V^{(n),\bullet})$ in $a_i C^{i}_Y(V^{(n),\bullet})$ is a coboundary. 

For a cocycle $x$, its projection in $C^{i}_Y(V_{A/a_iA}^{(n),\bullet})$ is a coboundary $d_i(y)$. Taking any lift $\tilde{y}$ of $y$ in $C^{i}_Y(V^{(n),\bullet})$, we see that $x-d_i(\tilde{y})$ is again a cocycle in $a_i C^{i}_Y(V^{(n),\bullet})$ by definition of $y$ thus a coboundary by the previous discussion. The same is then true for $x$ and we obtain the desired acyclicity. This concludes the proof of Lemma~\ref{flat-5}.


\bibliographystyle{alpha}
\bibliography{gm}





\end{document}